\documentclass[a4paper,11pt]{amsart}
\usepackage[latin1]{inputenc}
\usepackage{xcolor}

\usepackage{multicol}
\usepackage{amsfonts, amssymb, amsmath, amsthm, amscd,epsfig,mathrsfs}
\usepackage{latexsym}
\usepackage{graphicx}
\usepackage{cancel}
\usepackage[normalem]{ulem}
\usepackage[all]{xy}
\usepackage{enumerate}
\usepackage{verbatim}
\usepackage{hyperref}

\DeclareMathOperator{\Max}{\underline{Max}}
\DeclareMathOperator{\mult}{mult}
\DeclareMathOperator{\ord}{ord}

\DeclareMathOperator{\Sing}{Sing}
\DeclareMathOperator{\Spec}{Spec}

\newcommand{\Diff}{\mathit{Diff}}

\newcommand{\G}{{\mathcal G}}

\renewcommand{\L}{{\mathcal L}}

\newcommand{\Cont}{\mathrm{Cont}}

\newcommand{\Mm}{\mathrm{\underline{Max}\; mult}_X}

\newcommand{\Gn}{\G^{(n)}}

\newcommand{\Gd}{\G^{(d)}}

\newcommand{\Vn}{V^{(n)}}

\newcommand{\Vd}{V^{(d)}}

\newtheorem{Thm}{Theorem}[section]         
\newtheorem{Lemma}[Thm]{Lemma}

\newtheorem{Prop}[Thm]{Proposition}

\theoremstyle{definition}
\newtheorem{Parrafo}[Thm]{\ }
\newtheorem{Def}[Thm]{Definition}  
\newtheorem{Def-Prop}[Thm]{Definition-Proposition}

\theoremstyle{remark}
\newtheorem{Rem}[Thm]{Remark}

\newtheorem{Ex}[Thm]{Example}

\numberwithin{equation}{Thm}

\newcommand{\Mmp}{\mathrm{\underline{Max}\; mult}_{X'}}

\title{Contact loci and Hironaka's order}
\author{A. Bravo, S. Encinas, B. Pascual-Escudero}
\thanks{The authors were partially supported by MTM2015-68524-P; The first author  was partially  supported from the Spanish Ministry of Economy and Competitiveness, through the ``Severo Ochoa'' Programme for Centres of Excellence in R\&D (SEV-2015-0554)}


\AtEndDocument{\bigskip{\footnotesize
		\textsc{Depto. Matem\'aticas,
			Facultad de Ciencias, Universidad Aut\'onoma de Madrid 
			and Instituto de Ciencias Matem\'aticas CSIC-UAM-UC3M-UCM, Canto Blanco 28049 Madrid, Spain} \par  
		\textit{E-mail address}, A. Bravo: \texttt{ana.bravo@uam.es} \par

		\addvspace{\medskipamount}
		\textsc{Depto. Matem\'atica Aplicada,
			and IMUVA, Instituto de Matem\'aticas.
			Universidad de Valladolid} \par
		\textit{E-mail address}, S. Encinas: \texttt{santiago.encinas@uva.es}\par
		\addvspace{\medskipamount} 
		\textsc{Laboratoire des Sciences du Num\'erique de Nantes - Centrale Nantes} \par
		\textit{E-mail address}, B. Pascual-Escudero: \texttt{beatriz.pascual.escudero@gmail.com}\par
		Current address: \textsc{Department of Mathematical Sciences, University of Copenhagen, Universitetsparken 5, 2200 Copenhagen, Denmark -} 
		\texttt{beatriz@math.ku.dk}
}}

\begin{document}

\begin{abstract}
We study contact loci sets of arcs and the behavior of Hironaka's order function defined in constructive Resolution of singularities.
We show that this function can be read in terms of the irreducible components of the contact loci sets at a singular point of an algebraic variety. 
\end{abstract}

\maketitle

\tableofcontents

\section*{Introduction}

Resolution of singularities is a classical subject in algebraic geometry. Given an algebraic variety $X$, defined over a field $k$, the problem is to find a non-singular variety $\tilde{X}$ and a proper and birational morphism $f: \tilde{X} \to X$. The theorem of H. Hironaka (\cite{Hir}) asserts that a resolution of singularities exists when $k$ is a field of characteristic zero. Moreover, the theorem says that $f: \tilde{X} \to X$ can be defined as a composition of a finite number of blow ups at regular centers,     such that it induces an isomorphism on the non-singular locus of $X$, $X\setminus\Sing(X)$. The general problem, for varieties defined over arbitrary fields $k$, remains open, although we know that the answer is afirmative in low dimensions  (see for instance  \cite{Abhy1}, \cite{Abhy2}, \cite{Benito_V}, \cite{Cos_Pilt1,Cos_Pilt2}, \cite{Cut}, \cite{K_M1}, \cite{Lipman3}).  

\ 

The work of J. Nash on the theory of arc spaces was in part motivated by Hironaka's Theorem (cf. \cite{Nash}). A resolution of singularities of an algebraic variety $X$ may not be unique, and  one may wonder how much information about the process of resolution can be read on its space of arcs $\L(X)$.    There is a large number of papers where  arcs and singularities are studied. Just to mention a few see   \cite{Mou4}, \cite{Mus1}, \cite{Re3}, \cite{L-J_Mou_Re}, \cite{deF_Doc}, \cite{I_K}.  
 
\

This paper concerns the   study  of  an invariant    that is used in constructive resolution of singularities  and how it can be read in the space of arcs of a given variety. More precisely, we   explore how this invariant shows  up when considering the so called {\em contact loci with   a singular closed point   $\xi$}, say  $\Cont^{\geq n}(\mathfrak m_{\xi})$, i.e, the set of arcs that   have order at least $n$ at the maximal ideal $\mathfrak m_{\xi}$ of    $\xi$  for   $n\in {\mathbb N}$   (see \cite{E_L_M}, \cite{deF_E_I} and \cite{I08} where the structure of these sets is studied).  
 
\

\noindent {\bf Constructive resolution of singularities and Hironaka's order function}
\medskip

Hironaka's Theorem is existencial. A constructive resolution of singularities consists on describing a procedure to construct, step by step, a sequence of blow ups that leads to the resolution of a given variety $X$, 
\begin{equation}
\label{constructive}
X=X_0\leftarrow X_1 \leftarrow \ldots \leftarrow X_n=\tilde{X}.
\end{equation}
Constructive resolutions are given     in \cite{B-M}, \cite{V1}, and \cite{V2};
see also \cite{Br_E_V}, \cite{E_V97} and \cite{E_Hau}. 
Roughly speaking, to construct a sequence like (\ref{constructive}) one uses the so called {\em resolution functions defined on varieties}. These are upper semi-continuous functions 
 $$ \begin{array}{rrcl} 
 f_X:  &  X &  \to  & (\Lambda, \geq) \\
   & \xi & \mapsto &  f_X(\xi)
   \end{array}$$
 that are constant if and only if the variety is regular and    whose maximum value, $\max f_X$, achieved in a closed  regular   subset $\underline{\text{Max}} f_X$,  selects the  center to blow up. Thus  the sequence (\ref{constructive}) is defined so that 
$$\max f_{X_0}> \max f_{X_1} > \ldots > \max f_{X_n},$$
 where $\max f_{X_i}$ denotes the maximum value of $f_{X_i}$ for $i=0,1,\ldots, n$.  
Usually, $f_X$ is defined at each point as a sequence of rational numbers, the first set of coordinates    being the Hilbert-Samuel function at the point (see \cite{Br_E_V}) or the multiplicity  (see \cite{V}). Suppose that we are in this second case, and use the multiplicity as first coordinate of the resolution function $f_X$. Suppose in addition that $X$ is a variety of dimension $d$.  Then, the second coordinate of $f_X$ is the so called {\em Hironaka's order function in dimension $d$}, $\ord_X^{(d)}$ which is some positive rational number (see section~\ref{Hironaka_qpersistance}). The resolution function at a given point $\xi$ would be something like the following: 
\begin{equation}
\label{resol_coord}
f_X(\xi)=(\text{mult}_{{\mathfrak m}_{\xi}} {\mathcal O}_{X,\xi}, \ord_X^{(d)}(\xi),\ldots), 
\end{equation}
where $\text{mult}_{{\mathfrak m}_{\xi}} {\mathcal O}_{X,\xi}$ denotes   the multiplicity of the local ring $\mathcal{O}_{X,\xi}$ at the maximal ideal $\mathfrak{m}_{\xi}$. The remaining coordinates of $f_X(\xi)$ can be shown to depend on $\ord_X^{(d)}(\xi)$, thus, we usually say that this rational number is the main invariant in constructive resolution. 
\medskip

In \cite{Br_E_P-E} we showed that   $\ord_X^{(d)}(\xi)$,  can be read from the set of arcs  with center $\xi$, $\L(X, {\xi})$. To this end   we worked with the so called {\em Nash multiplicity sequences   of arcs} introduced by M. Lejeune-Jalabert in \cite{L-J} for  the case of a germ of a point of a  hypersurface, and generalized afterwards by H. Hickel in \cite{Hickel05}. For a given point $\xi$ in a variety $X$,  these sequences of numbers are  intrinsic, and only depend on the  set  $\L(X, {\xi})$. 
\medskip

Finally we point out that the  invariant $\ord_X^{(d)}(\xi)$ can also be defined if $k$ is a perfect field of positive characteristic; only, it is too coarse and it does not provide enough information to be able to construct  a  resolution function.     In \cite{BEP2}  we showed that the results in  \cite{Br_E_P-E}  can also be extended to this case, therefore providing a geometrical meaning to  Hironaka's order function in positive characteristic. 
\bigskip

\noindent{\bf Nash multiplicity sequences:  the {\em persistance} and the {\em ${\mathbb Q}$-persistance}}
\medskip

Suppose $X$ is a singular variety of maximum multiplicity $m>1$. Then given a point $\xi\in \Sing(X)$ of multiplicity $m$, and an arc $\varphi \in \L(X, \xi)$,  the {\em sequence of Nash multiplicities of $\varphi$} is a non-increasing sequence of integers, 
\begin{equation}
\label{primera_Nash}
m=m_0\geq m_1\geq m_2 \geq \ldots
\end{equation}
where $m_0=m$ is the multiplicity at the point $\xi$, and the rest of the numbers in the sequence can be interpreted as a {\em refinement of the ordinary multiplicity at $\xi$ along the arc $\varphi$}  (see the discussion in Sections \ref{RAandNash} and \ref{Hironaka_qpersistance}).  
\medskip

Suppose that $\varphi$ is a $K$-arc, with $K\supset k$, which gives a morphism 
$\varphi:{\mathcal O}_{X,\xi}\to K[[t]]$. 
When the generic point of $\varphi$ is not contained in the stratum of multiplicity $m$ of $X$, then there is some subindex $l\ge 1$ in sequence (\ref{primera_Nash})  for which $m_l<m_0$. We will be interested in the first subindex for which the inequality holds and call it {\em the persistance of the arc $\varphi$}, $\rho_{X,\varphi}$. To eliminate the impact of the order of the arc at the point, we will normalize the persistance setting 
\begin{equation}
\label{rho_norm_1}
\overline{\rho}_{X,\varphi}=\frac{\rho_{X,\varphi}}{\nu_t(\varphi)},
\end{equation}
 where $\nu_t(\varphi)$ denotes the order of the image by $\varphi$, of the maximal ideal of $\xi$, i.e., $\varphi({\mathfrak m}_{\xi})$, at the regular local ring $K[[t]]$. We will work simultaneously with another invariant which is a refinement of the persistance:  the {\em ${\mathbb Q}$-persistance}, which we denote  by  $r_{X,\varphi}$,  and its normalized version $\overline{r}_{X,\varphi}$. In fact, the two invariants are related since for a given arc $\varphi$ it can be shown that 
\begin{equation} 
\label{rho_r} 
\rho_{X,\varphi}= \lfloor r_{X,\varphi}\rfloor  \ \text{ and } \    r_{X,\varphi} =\frac{1}{\nu_t(\varphi )}\cdot \lim _{n\rightarrow \infty }\frac{\rho _{X,\varphi _{n}}}{n} \in {\mathbb Q}_{\geq 1},
\end{equation}
where for each $n\geq 1$, $\varphi _n=\varphi \circ i_n$  and $i_n^*:K[[t]]\longrightarrow K[[t^n]]$ maps $t$ to $t^n$. 
\medskip

In what follows we will denote by $\Mm$ the (closed) set of points of maximum multiplicity $m$ of $X$. With this notation, in \cite{Br_E_P-E}, \cite{BEP2} we proved the following theorem: 
\medskip

\begin{Thm}\label{principal} \cite[Theorem 3.6]{Br_E_P-E}, \cite[Theorem 6.1]{BEP2} 
 Let $X$ be a $d$-dimensional   algebraic  variety defined over a perfect field $k$, and let $\xi\in \Mm$. Then 
\begin{equation}\label{desigualdad_i}
\ord_X^{(d)}(\xi)\leq \inf_{\varphi\in {\L(X, {\xi})}}\{\overline{r}_{X,\varphi}\}=\inf_{\varphi\in \L(X,{\xi})}\left\{\frac{1}{\nu_t(\varphi)}\lim_{n\to \infty}\frac{\rho_{X,\varphi_{n}}}{n}\right\}.
\end{equation}
Moreover, the infimum is a minimum, i.e., there is some arc $\eta\in \L(X,{\xi})$ such that:  
\begin{equation}
\label{igualdad_divisorial_i}
\ord_X^{(d)}(\xi)=\overline{r}_{X,\eta}=\frac{1}{\nu_t(\eta)}\lim_{n\to \infty}\frac{\rho_{X,\eta_{n}}}{n}.
\end{equation}
\end{Thm}
\bigskip
 
\noindent{\bf Results}
\medskip
 
The purpose of this paper is to study the behaviour of the normalized ${\mathbb Q}$-persistance,   $\overline{r}_X$, as a function 
 on $\L(X, {\xi})$.   Observe that,  from the way $\overline{r}_{X}$  is defined, it will not be an upper-semi continuous function. 
 
One may wonder, for instance,  if  equality (\ref{igualdad_divisorial_i}) holds generically at   $\L(X, {\xi})$.
This we do not know, and do not expect it either.   Thus we formulate our question in a slightly different way by selecting suitable closed sets in $\L(X, {\xi})$.  

 Recall  that if   ${\mathfrak a}$ is a sheaf of ideals on $X$, then, for each $n\in {\mathbb Z}_{\geq 1 }$ one can define the closed subset   of $\L(X)$:
$$\text{Cont}^{\geq n}({\mathfrak a}):=\{\varphi\in \L(X): \nu_t(\varphi({\mathfrak a}))\geq n\},$$
and the locally closed set
$$\text{Cont}^{=n}({\mathfrak a}):=\{\varphi\in \L(X): \nu_t(\varphi({\mathfrak a})) = n\}.$$

See Definition \ref{contacto_ideal} below. With this notation, we  show: 

\medskip

\noindent{\bf Proposition \ref{proposicion_abierto}.} 
{\em
Let $X$ be a $d$-dimensional algebraic   variety defined over a perfect field $k$,   and let $\xi\in \Mm$. 
Suppose there is some $s\geq 1$  and an arc $\varphi_0\in \text{Cont}^{=s}({\mathfrak m}_{\xi})$ with $\overline{r}_{X,{\varphi_0}}=\ord_{\xi}^{(d)}(X)$.
Then there is a non-empty open subset ${\mathfrak W}$ of $\text{Cont}^{ \geq s}({\mathfrak m}_{\xi})$,
containing $\varphi_0$, such that for all arcs $\varphi \in {\mathfrak W}$,  $\overline{r}_{X,\varphi}=\ord_{\xi}^{(d)}(X)$. If, in addition,  the generic point of $\varphi_0$   is not contained in $\Sing (X)$  and the characteristic of $k$ is zero, then  there are fat (divisorial) arcs in ${\mathfrak W}$.
}
\medskip

It is natural to ask for which   values of $s$  the previous proposition holds. Observe that, since $\xi\in X$ is a singular point, it may happen that:
$$\L(X, {\xi})=\text{Cont}^{\geq 1}(\mathfrak{m}_{\xi})=\text{Cont}^{\geq 2}(\mathfrak{m}_{\xi})=\ldots=\text{Cont}^{\geq t_0}(\mathfrak{m}_{\xi})\supsetneq \text{Cont}^{\geq t_0+1}(\mathfrak{m}_{\xi})\supseteq\ldots$$
and it would be interesting to know whether the statement is valid    for $s=t_0$, the minimum order of an arct at $\xi$. We do not know how to compute the value $t_0$, but we can find  values for which the proposition holds   by looking at the normalized blow up of $X$ at $\xi$,  $X\longleftarrow\overline{X_1}$. Observe that in this setting, after removing a closed set of codimension at least two in $\overline{X_1}$, we can restrict to an open set $U$ such that we have a log resolution of the maximal ideal of the point,  $\mathfrak{m}_{\xi}$:
\begin{equation}
\label{def_c_i_intro}
\mathfrak{m}_{\xi}\mathcal{O}_{U}=I(H_1)^{c_1}\cdots I(H_{\ell})^{c_{\ell}}
\end{equation}
where the hypersurfaces $H_i$ are irreducible and have only normal crossing in $U$. In fact  the number $c:=\text{min}\{c_1,\ldots, c_{\ell}\}$ is an upper bound for $t_0$. 

\medskip 

\noindent{\bf Proposition \ref{corolario_abierto}.}
{\em
Let $X$ be a $d$-dimensional algebraic   variety defined over a perfect field $k$ and  let $\xi\in \Mm$.
Then for every  $n\geq 1$  and every $c_i$ as in (\ref{def_c_i_intro}), $i=1,\ldots, \ell$,  there is a non-empty open set 
${\mathfrak U}_{nc_i}\subseteq\Cont^{\geq nc_i}(\mathfrak{m}_\xi)$  such that for all
$\varphi \in {\mathfrak U}_{nc_i}$, $\overline{r}_{X,\varphi}=\ord_{X}^{(d)}(\xi)$.
}

\

In particular, for those cases in which $c=t_0$, the statement says that there is a non-empty open set ${\mathfrak U}\subset \L(X, {\xi})$ such that for all arcs in ${\mathfrak U}$ the equality (\ref{igualdad_divisorial_i}) holds. 
\medskip

In \cite{E_L_M} and \cite{deF_E_I} it was shown that  if $X$ is a complex algebraic variety, then for each $n$,   the closed subsets $\text{Cont}^{\geq n}({\mathfrak a})$ have a finite number of (fat) irreducible components and that, moreover, these fat irreducible components are maximal divisorial sets. 
\medskip

Here we study the behaviour of the normalized ${\mathbb Q}$-persistance  on the irreducible fat components of  $\text{Cont}^{\geq n}({\mathfrak m}_{\xi})$. On the one hand we show that for certain values of  $n$,  equality (\ref{igualdad_divisorial_i}) always holds for the generic point of some irreducible fat component: 
\medskip

\noindent{\bf Theorem \ref{componentes}.} {\em 
Let $X$ be a $d$-dimensional  algebraic  variety defined over a perfect field $k$, let $\xi\in \Mm$, and let  $\{T_{\lambda_m}\}_{\lambda_m\in \Lambda_m}$  
be the fat irreducible components of   $\text{Cont}^{\geq m}({\mathfrak m}_{\xi})$, 
with generic points  $\{\Psi_{\lambda_m}\}_{\lambda_m\in \Lambda_m}$ for $m\geq 1$.   If $m=nc_i$ for some $n\geq 1$ and some $c_i$ as in (\ref{def_c_i_intro})  then 
$$\ord_X^{(d)}(\xi) =\min \{\overline{r}_{X,\Psi_{{\lambda}_m}}: \lambda_m\in \Lambda_m\}.$$ 
In addition, if $k=\mathbb{C}$ then equality (\ref{igualdad_divisorial_i}) holds 
at the generic point of a maximal divisorial set.
}
\medskip

In particular, for those cases in which $c=t_0$, the statement says that the equality (\ref{igualdad_divisorial_i}) holds at the generic point of a fat irreducible component of  $\L(X, {\xi})$.

It is quite natural to investigate if the same statement  holds for the fat irreducible components of $\text{Cont}^{\geq n}({\mathfrak m}_{\xi})$ for   arbitrary values of  $n$,   but Example \ref{ejemplo_no} indicates  that this is not the case.  However, it can be proved  that equality (\ref{igualdad_divisorial_i})  holds assymptotically when $n$ is arbitrary large: 
\medskip

\noindent{\bf Theorem \ref{componentes_limite}.} {\em Let $X$ be a $d$-dimensional algebraic variety defined over a perfect field $k$, and   let $\xi\in \Mm$.
For each $m\in {\mathbb N}$, let $\{T_{m,\lambda_m}\}_{\lambda_m\in \Lambda_m}$ be the fat irreducible components of $\Cont^{\geq m}({\mathfrak m}_{\xi})$ and let 
	$\Psi_{m, \lambda_m}$ be the generic point of $T_{m,\lambda_m}$ for    $\lambda_m\in \Lambda_m$.
	For each $m\geq 1$ set:
	\begin{equation*}
	\delta_{m}:=\inf\left\{\bar{r}_{\Psi_{m,\lambda_m}}\mid \lambda_{m} \in \Lambda_m \right\}.
	\end{equation*}
	Then we have that
	\begin{equation*}
	\ord^{(d)}_{X}(\xi)=\lim_{m\to\infty} \delta_m.
	\end{equation*}
}
\medskip

In the last section of this paper we  explore the possible values of the normalized ${\mathbb Q}$-persistance   when $X$ is defined over a field of characteristic zero (resolution of singularities is needed for these results). On the one hand we show that, regarding the study of the values of $\overline{r}_X$,  it suffices to study fat divisorial arcs: 
\medskip

\noindent{\bf Theorem \ref{ThExistDivFat}.}  {\em Let $X$ be a $d$-dimensional algebraic variety defined over a field of characteristic zero.
Fix a point $\xi\in \Mm$ and let
$\varphi \in \mathcal{L}(X, {\xi})$ be an arc such that $\varphi\not\in\mathcal{L}(\Sing(X))$.
Then there exists a divisorial fat arc $\psi\in\L(X, {\xi})$ such that 
	\begin{itemize}
		\item  $\varphi\in \overline{\{\psi\}}$   and 
		\item $\bar{r}_{X,\varphi}=\bar{r}_{X,\psi}$.
	\end{itemize}}
\medskip

Finally in  \cite{P-E2}   it is proven that  $\xi$ is an isolated point of $\Mm$   if and only if the set $\left\{\overline{r}_{X,\varphi}\right\}_{\varphi\in \L(X, {\xi})}$ has an upper bound.  Here  we give more accurate bounds for   $\overline{r}_X$ on $\L(X, {\xi})$ in the isolated case. To state this result, we use the fact that it is possible to associate a (canonical) Rees algebra with
the set $\Mm$, say $\G_{X'}$,  in an (\'etale) neighborhood of $\xi$, say $X'\to X$ (see \ref{setting}).
\medskip

\noindent{\bf Theorem  \ref{Values}.} {\em Let $X$ be a $d$-dimensional algebraic variety defined over a field of characteristic zero $k$ and let $\xi\in \Mm$. Let $\mu: X'\to X$ be an \'etale morphism   with $\mu(\xi')=\xi$ where $\G_{X'}$ is defined, and 
assume that,  up to integral closure, $\mathcal{G}_{X'}={\mathcal O}_{X'}[IW^b]$ (see (\ref{salvo_entera})).
Let $\Pi:Y\to X'$ be a simultaneous log-resolution of the ideals $I$ and $\mathfrak{m}_{\xi'}$.
Denote by  $H_1,\ldots,H_N$ the irreducible components of the exceptional locus,
\begin{equation} 
	I\mathcal{O}_{Y}=I(H_1)^{a_1}\cdots I(H_N)^{a_N},
	\qquad
	\mathfrak{m}_{\xi'}\mathcal{O}_{Y}=I(H_1)^{c_1}\cdots I(H_N)^{c_N}. 
\end{equation} 
Set $\Lambda=\left\{i\in \left\{1,\ldots,N\right\}\mid a_i\neq 0\right\}$.
Then, for any arc $\varphi$ in $\L(X, {\xi})$ with $\varphi\not\in\L(\Sing(X))$, 
\begin{equation*}
	\frac{1}{b}\min_{i\in\Lambda}\frac{a_i}{c_i}\leq \bar{r}_{X,\varphi}\leq \frac{1}{b}\max_{i\in\Lambda}\frac{a_i}{c_i}, 
\end{equation*}
where we use the convention that $\frac{a_i}{c_i}=\infty$ whenever $c_i=0$ and $a_i\neq 0$.   Moreover, 
\begin{equation*}
	\frac{1}{b}\min_{i\in\Lambda}\frac{a_i}{c_i}=\inf\left\{ \bar{r}_{X,\varphi} \mid \varphi\in\L(X, {\xi}) \right\}
	\quad\text{and}\quad
	\frac{1}{b}\max_{i\in\Lambda}\frac{a_i}{c_i}=\sup\left\{ \bar{r}_{X,\varphi} \mid \varphi\in\L(X, {\xi}) \right\}.
\end{equation*}
}

\medskip

\noindent{\bf On the organization of the paper} 
\medskip

The paper is organized as follows. 
Section \ref{Jet_Arcs} introduces concepts and definitions about arc spaces, contact loci sets and divisorial sets. 
Nash multiplicity sequences and the persistance are defined in section~\ref{RAandNash}.
Section~\ref{Rees_Algebras} is devoted to Rees Algebras, here we define the natural order function of a Rees Algebra.
Section~\ref{Presentaciones_locales} introduces local presentations, which allow to express  the maximum stratum of the multiplicity function in terms of a Rees Algebra.
Hironaka's order function is presented in section~\ref{Hironaka_qpersistance}.

Results are   stated and proved  in sections~\ref{genericos}, \ref{irreducible_Hironaka} and \ref{valores}.
\medskip

{\em Acknowledgements.} We profited from conversations with C. Abad, A. Benito and O. E. Villamayor.
We would like to thank S. Ishii and T. Yasuda for useful comments.
Finally, we also want to thank the referee for careful reading and useful suggestions.


\section{Arcs,  valuations and contac loci} \label{Jet_Arcs}

\begin{Def}
	Let $Z$ be a  scheme over a field $k$, and let $K\supset k$ be a field extension.
	An {\em m-jet  in   $Z$} is a morphism $\vartheta: \text{Spec}\left(K[[t]]/\langle t^{m+1}\rangle\right)\to Z$ for some  $m\in {\mathbb N}$. 
\end{Def}
If ${\mathcal S}ch/k$ denotes the category of $k$-schemes and ${\mathcal S}et$ the category of sets, then the contravariant functor:
$$\begin{array}{rcl} 
{\mathcal S}ch/k & \longrightarrow & {\mathcal S}et\\
Y  & \mapsto & \text{Hom}_k(Y\times_{\text{Spec}(k)}\text{Spec}(k[[t]]/\langle t^{m+1}\rangle), Z)
\end{array}
$$
is representable by a $k$-scheme  ${\mathcal L}_m(Z)$, the {\em space of $m$-jets} over $Z$. If $Z$ is of finite type over $k$, then so is  ${\mathcal L}_m(Z)$ (see \cite{Vojta}). For each pair $m\geq m'$ there is the (natural) truncation  map ${\mathcal L}_{m}(Z) \to {\mathcal L}_{m'}(Z)$. In particular, for $m'=0$, ${\mathcal L}_{m'}(Z)=Z$ and we will denote by ${\mathcal L}_m(Z, \xi)$  the fiber of the (natural) truncation map over a point $\xi\in Z$. Finally, if $Z$ is smooth over $k$ then ${\mathcal L}_m(Z)$ is also smooth over $k$ (see \cite{I2}).

By taking the inverse limit of the ${\mathcal L}_m(Z)$,  the {\em arc space of $Z$} is defined,  
$${\mathcal L}(Z):=\lim_{\leftarrow}{\mathcal L}_m(Z).$$ This is the scheme representing the functor (see \cite{Bhatt}): 
$$\begin{array}{rcl} 
{\mathcal S}ch/k & \longrightarrow & {\mathcal S}et\\
Y  & \mapsto & \text{Hom}_k(Y\tilde{\times}\text{Spf}(k[[t]]), Z).
\end{array}
$$

A $K$-point in $\L(Z)$ is an  {\em arc of $Z$} and can be seen as   a morphism $\varphi: \text{Spec}(K[[t]])\to Z$ for some $K\supset k$.  The  image by $\varphi$ of the closed point   is called the {\em center of the arc $\varphi$}.  If the center of $\varphi$ is $\xi\in Z$ then it induces a $k$-homomorphism  
${\mathcal O}_{Z,\xi}\to K[[t]]$ which we will denote by $\varphi$ too; in this case the image by $\varphi$ of the maximal ideal,  $\varphi({\mathfrak m}_{\xi})$,  generates an ideal $\langle t^m\rangle\subset K[[t]]$ and  then we will say that {\em the order of $\varphi$ is $m$} and we will denote it by $\nu_t(\varphi)$.   We will denote  by $\L(Z, {\xi})$ the set of arcs in  $\L(Z)$ with center $\xi$. The {\em generic point of $\varphi$ in $Z$} is the point in $Z$ determined by the  $\text{kernel}$ of $\varphi$.

\begin{Def}
	An arc 	$\varphi: \text{Spec}(K[[t]])\to Z$ is {\em thin} if it factors through a  proper closed subscheme of $Z$. Otherwise we say that $\varphi$ is {\em fat}. An irreducible closed subset $C\subset {\mathcal L}(Z)$ is said to be a {\em fat closed subset} if its generic point is a fat arc. Otherwise $C$ is said to be {\em thin}. 
\end{Def}
\bigskip

\noindent{\bf Divisorial arcs and maximal divisorial sets}
\medskip

\noindent In the following lines we will assume that $X$ is an (irreducible) algebraic variety defined over a field $k$ and will denote by $K(X)$ its quotient field.
\medskip

Observe that any fat arc  $\varphi: \text{Spec}(K[[t]])\to X$  defines a discrete valuation  on   $X$. This is the {\em valuation corresponding to $\varphi$}, $v_{\varphi}$. If $\varphi$ is thin, then it defines a valuation in the quotient field $K(Y)$ of some  (irreducible) subvariety $Y\subset X$. On the other hand, note that for any discrete valuation $v$ of  $K(X)$  one can  define a (non-necessarily unique) arc $\varphi: \Spec{K}[[t]] \to X$, for a  suitable field  $K\supset k$,  whose corresponding valuation is $v$. 

\begin{Def}  We say that a divisor $D$  is a  {\em divisor over $X$} if there is a proper and birational morphism from a normal variety, $X'\to X$, so that $D$ is a divisor on $X'$.  We say that a fat arc    $\varphi\in {\mathcal L}(X)$ is {\em divisorial} if  the (discrete) valuation defined by $\varphi$, $v_{\varphi}$, is a multiple of the valuation defined by some divisor over $X$, i.e., if there is some $q\in {\mathbb N}$ and some divisor $D$ over $X$ such that $v_{\varphi}=q\text{val}_D$. 
\end{Def}
The divisorial fat arcs of a variety $X$ form a   subset of  the fat arcs defined on it. We refer to \cite{I_Crelle} for discussions and examples regarding  this matter. 

\begin{Def}   An irreducible closed subset $C\subset {\mathcal L}(X)$ is said to be a {\em divisorial closed subset} if its generic point is a divisorial (fat) arc. 
\end{Def}

\begin{Def}\label{def_maximal_divisorial}\cite[Definition 2.8]{I08} 
Given a divisorial valuation $v$ over a variety $X$, the {\em maximal divisorial set corresponding to $v$} is defined as: $$C_X(v):=\overline{\{\varphi \in  {\mathcal L}(X): v_{\varphi}=v \}},$$ where $\overline{\{ \ \}}$ denotes the Zariski closure in  ${\mathcal L}(X)$. 
\end{Def}
\bigskip

\pagebreak

\noindent{\bf Contact loci and (maximal) divisorial sets}
\medskip

\begin{Def}\label{contacto_ideal} (\cite{E_L_M},\cite{I08}) Let ${\mathfrak a}$ be a sheaf of ideals on $X$. Then  one can define: 
	\begin{equation}\label{contm}
 	\text{Cont}^{m}({\mathfrak a}):=\{\varphi\in {\L}(X): \nu_t(\varphi({\mathfrak a}))=m\}
	\end{equation}
and  
\begin{equation} \label{contmm}
\text{Cont}^{\geq m}({\mathfrak a}):=\{\varphi\in {\L}(X): \nu_t(\varphi ({\mathfrak a}))\geq m\},
\end{equation}
where, if $\varphi: \text{Spec}(K[[t]])\to X$, and if $U\subset X$ is an affine open set containing the center of $\varphi$, then  $\nu_t(\varphi({\mathfrak a}))$ is defined as the usual order at $K[[t]]$ of the ideal $\varphi(\Gamma(U,{\mathfrak a}))$.  For $m\in {\mathbb N}$, the subsets $\text{Cont}^{\geq m}({\mathfrak a})$ are closed and the  $\text{Cont}^{m}({\mathfrak a})$ are locally closed in $ {\L}(X)$.  If $Y\subset X$ is a closed subscheme of $X$ defined by a sheaf of ideals ${\mathfrak a}$ then one can also define 
$$\text{Cont}^{m}(Y):=\text{Cont}^{m}({\mathfrak a}), \  \ \text{ and } \ \ \text{Cont}^{\geq m}(Y):=\text{Cont}^{\geq m}({\mathfrak a}).$$
\end{Def} 

In the following paragraphs we recall some results from \cite{E_L_M}, \cite{deF_E_I} and \cite{I08} regarding the expression of the subsets (\ref{contm})  and (\ref{contmm}) in terms of irreducible components in the space of arcs of $X$ and their  connection    with the notion of maximal divisorial sets from Definition \ref{def_maximal_divisorial}. 
\medskip

Suppose   now that  $X$ is  a smooth complex variety and let $E=\sum_{i=1}^tE_i$ be a simple normal crossing divisor on $X$. Given a multi-index $\nu=(\nu_i)\in {\mathbb N}^t$, define the support of $\nu$ to be
$$\text{supp}:=\{i\in [1,t]: \nu_i\neq 0\}$$
and 
$$E_{\nu}=\cap_{i\in \text{supp}(\nu)}E_i.$$
Then $E_{\nu}$ is either empty or a smooth subvariety of $X$.   Assume that $E_\nu$ is connected. For a multi-index $\nu\in {\mathbb N}^t$ and an integer $m\geq \text{max}_i \{\nu_i\}$, consider the {\em multi-contact} loci: 
\begin{equation}
\label{Multi_Contact}
\text{Cont}^{\nu}(E)_m=\{\sigma\in {\mathcal L}_{m}(X): \nu_t(\sigma(E_i))=\nu_i, 1\leq i\leq t\},
\end{equation}
and the corresponding subset $\text{Cont}^{\nu}(E)\subset {\mathcal L}(X)$. Provided that $E_{\nu}\neq \emptyset$ it can be checked that  $\text{Cont}^{\nu}(E)_m$ is a smooth irreducible locally closed subset of ${\mathcal L}_{m}(X)$ (see \cite[\S 2]{E_L_M}). Furthermore, $\text{Cont}^{\nu}(E)$  it is a maximal divisorial set (see \cite[Corollary 2.6]{E_L_M} and \cite[Proposition 2.12]{deF_E_I}).  The following theorem asserts that the fat irreducible components of $ \text{Cont}^{m}({\mathfrak a})$ for a given sheaf of ideals in $X$ can be computed via a log-resolution of the ideal:

\begin{Thm}\cite[Theorem 2.1]{E_L_M} \label{ThUnionCont}
Let $X$ be a smooth complex  variety and let ${\mathfrak a}\subset {\mathcal O}_X$ defining a subscheme $Z\subset X$. Let $\Pi: Y\to X$ be a log resolution of $X$ with $E=\sum_{i=1}^tH_i$ a simple normal crossing divisor on $Z$ with 
	\begin{equation}
	\label{componentes_log}
	{\mathfrak a}{\mathcal O}_{Y}={\mathcal O}_{Y}\left(-\sum _{i=1}^tr_iH_i\right).
	\end{equation}
	Then for every positive integer $p$, we have a disjoint union 
	\begin{equation} \label{UnionCont}
	\bigsqcup_{\nu}\Pi_{\infty}(\Cont^{\nu}(E))	\subset \Cont^{p}(Z),
	\end{equation}
where the union is taken over those $\nu\in {\mathbb N}^t$ such that $\sum_{i}\nu_ir_i=p$, and the complement in  $\Cont^{ p}(Z)$ of the above union is thin. 
\end{Thm}	

The next results  generalize the previous theorem to the context of non-necessarily smooth complex varieties, and moreover   relate the expressions in Theorem \ref{ThUnionCont} to the  maximal divisorial sets of valuations that dominate the sheaf of ideals ${\mathfrak a}$: 
 
\begin{Prop}\cite[Proposition 3.4]{I08}
	Let $v=q\cdot\text{val}_D$ be a divisorial valuation over a (non-necessarity smooth) variety $X$. Let $\Pi: Y\to X$ be a resolution of singularities of $X$ such that the irreducible divisor $D$ appears on $Y$. Then,
	$$C_X(v)=\overline{\Pi_{\infty}(\text{Cont}^{\ q}(D))}.$$ 
	In particular, $C_X(v)$ is irreducible.
\end{Prop}	

\begin{Prop}\cite[Proposition 2.12]{deF_E_I} \label{maximal_divisorial} Let $X=\text{Spec}(A)$ be a (non necessarily smooth) affine complex variety and let ${\mathfrak a}$ be a non-zero sheaf of ideals. Then for any $m\in {\mathbb N}$   the number of fat irreducible components of $\text{Cont}^{\geq m}({\mathfrak a})$ is finite, and every fat irreducible component  is a maximal divisorial set. 
\end{Prop}

\bigskip

\noindent{\bf Arc spaces and \'etale morphisms}
\medskip

\noindent As our results are of local nature we will be assuming that $X$ is an affine algebraic variety. In addition,
most of the arguments used in the proofs along sections \ref{genericos}, \ref{irreducible_Hironaka} and \ref{valores} are first proven in an \'etale neighborhood of a point $\xi \in X$. Thus we    include here a few comments concerning the behavior of arcs up to \'etale morphisms. In the following lines we will be assuming that    $\mu:X'\to X$ is an \'etale morphism with $\mu(\xi')=\xi$, for some $\xi'\in X'$.

\begin{Rem}\label{rem:lifting_arc_etale_open}  
By \cite[Proposition 5.9]{Vojta}, we have that
$\L(X')=\L(X)\times_X X'$.
As a consequence,
the induced morphism $\mu_{\infty}: \L(X')\to \L(X)$ is \'etale (locally of finite type), therefore flat  (see \cite[\S 8.5, Proposition 17]{Bosch}), and hence open. 
\end{Rem}

\begin{Rem}\label{rem:lifting_arc_etale}
Let $\varphi:  \Spec(K[[t]])  \to X$  be an arc with center $\xi$.
Since $\mu:X'\to X$ is \'etale, then it is formally \'etale  (see \cite[Chapter I, Remark 3.22]{Milne}) and therefore there is a lifting $\varphi'$ with center $\xi'$,  $\varphi': \Spec(K'[[t]]) \to X'$, where $K'$ is a separable extension of $K$.
Another way to prove the same is by using that $\L(X')=\L(X)\times_X X'$ (Remark \ref{rem:lifting_arc_etale_open}).
In any case one gets    a commutative diagram: 
\begin{equation*}
\xymatrix{
\Spec(K'[[t]]) \ar[d]  \ar[r]^{\qquad\varphi'}   &  X' \ar[d]^{\mu} \\
\Spec(K[[t]])  \ar[r]^{\qquad\varphi}  & X.
}
\end{equation*}
In particular,  $\mu_{\infty}(\varphi')=\varphi$. 

\end{Rem}

\begin{Rem} \label{Rem:Blup_etale}
With the setting as in  Remark   \ref{rem:lifting_arc_etale},
let $\pi :Y\longrightarrow X$ be the blow up of $X$ at $\xi $ and  let $\varphi _Y$ be the lifting of $\varphi $ to $Y$ (provided that $\varphi$ is not constant). We have the following commutative diagram,

\begin{equation*}
\xymatrix@R=15pt@C=50pt{
	\Spec(K'[[t]]) \ar[dd]_{\varphi '} \ar[ddr]^{\varphi _Y'} & \\
	& & & & \\
	X' \ar[d]_{\mu } & Y' \ar[l]_>>>>>{\pi' } \ar[d]_{\mu _Y} \\
	X & Y \ar[l]_>>>>>{\pi } \\
	& & & & \\
	\Spec(K[[t]]) \ar[uu]^{\varphi } \ar[uur]_{\varphi_Y} & 
}\end{equation*}
where $Y'=Y\times_X X'$,   $\mu _Y$ is \'etale and $\varphi '_Y$ is the lifting of $\varphi '$ to $Y'$.
Note that $Y'$ is the blowup at $\mu^{-1}(\xi)$.
Let $\xi_Y$ be the center of $\varphi_Y$. If $\xi '_Y$ is the center of the arc $\varphi '_Y$, then $\mu _Y(\xi '_Y)=\xi _Y$.
\end{Rem}

\begin{Lemma} \label{Lema_Etale_Arco}
Let  $\varphi,\psi\in\mathcal{L}(X, {\xi})$ be  two arcs  such that $\varphi\in\overline{\{\psi\}}$. Assume that $\varphi'\in \mathcal{L}(X', {\xi'})$ is an arc with $\mu_{\infty}(\varphi')=\varphi$. Then there is an arc $\psi'\in \mathcal{L}(X',{\xi'})$ with $\varphi'\in\overline{\{\psi'\}}$ and such that  $\mu_{\infty}(\psi')=\psi$.
\end{Lemma}

\begin{proof}
	
The morphism   $\mu_{\infty}:\mathcal{L}(X')\to\mathcal{L}(X)$ is  flat  (see Remark \ref{rem:lifting_arc_etale_open}). From here   it can be checked that there exists an arc  $\psi'\in\mathcal{L}(X', {\xi'})$ with  $\varphi'\in\overline{\{\psi'\}}$ and $\mu_{\infty}(\psi')=\psi$  (see \cite[Chapter 1, Corollary 2.8]{Milne}). 
\end{proof}

\section{Nash multiplicity sequences, the  persistance, and the ${\mathbb Q}$-persistance}\label{RAandNash}

In this section we will recall the notion of    {\em Nash multiplicity sequence} along an arc of a variety $X$. This will lead us to define an invariant for each arc $\varphi$ with center a given point $\xi\in X$: {\em the persistance}, and a refinement, the {\em ${\mathbb Q}$-persistance}.  
\bigskip

\noindent {\bf Nash multiplicity sequences}
\medskip

\noindent Let $X$ be an algebraic variety defined over a   perfect field $k$    and  let  $\xi \in X$ be a (closed) point.   Assume that $X$ is locally  a hypersurface in a neighborhood of $\xi$,  $X\subset V$, where   $V$ is smooth over $k$,  and work at the completion $\widehat{\mathcal O}_{V,{\mathfrak m}_{\xi}}$. Under these hypotheses,  in \cite{L-J}, Lejeune-Jalabert introduced the   {\em Nash multiplicity sequence  along an arc $\varphi\in {\mathcal L}(X, {\xi})$}  (in fact, the hypotheses in \cite{L-J} are weaker, but we are interested in working over perfect fields).   The Nash multiplicity sequence of $X$ along $\varphi$ is a non-increasing sequence of non-negative integers 
\begin{equation}
\label{Nash_sequence}
m_0\geq m_1\geq \ldots \geq m_l=m_{l+1}=...\geq 1, 
\end{equation}
where  $m_0$ is the usual multiplicity of $X$ at $\xi$, and the rest of the terms  are computed by considering  suitable stratifications on ${\mathcal L}_{m}(X, {\xi})$  defined via the action of certain differential operators on the fiber of the jets spaces  ${\mathcal L}_{m}(\text{Spec}(\widehat{\mathcal O}_{V,{\mathfrak m}_{\xi}}))$ over $\xi$ for $m\in {\mathbb N}$.  The sequence (\ref{Nash_sequence})   can be interpreted   as  the {\em multiplicity of $X$ along the arc $\varphi$}: thus it can be seen as a refinement of the usual multiplicity.   The sequence stabilizes at the value given by the multiplicity $m_l$ of  $X$ at the generic point of the arc $\varphi $ in $X$ (see \cite[\S 2, Theorem 5]{L-J}).  

\vspace{0.2cm}

In \cite{Hickel05}, Hickel generalized Lejeune-Jalabert's construction  to the case of an arbitrary variety $X$ and presented   the sequence (\ref{Nash_sequence}) in a different way  which we will explain along the following lines.

\vspace{0.2cm}

Since the  arguments are of local nature, let us suppose that  $X=\text{Spec}(B)$  is   affine.
Let $\xi\in  X$ be a point (which we may assume to be closed) of multiplicity $m_0$, and let $\varphi $ be an arc in $X$ centered at $\xi $. Consider   the natural morphism   
$$\Gamma _0=\varphi \otimes i:B\otimes_k k[t]\rightarrow K[[t]]\mbox{,}$$
which is additionally an arc in $X_0=X\times \mathbb{A}^1_k$ centered at the point $\xi _0=(\xi ,0)\in X_0$. This arc determines    a sequence of blow ups at points:
\begin{equation}\label{intro:diag:Nms}
\xymatrix@R=15pt@C=50pt{
	\mathrm{Spec}(K[[t]]) \ar[dd]^{\Gamma _0} \ar[ddr]^{\Gamma _1} \ar[ddrrr]^{\Gamma _l} & & & & \\
	& & & & \\
	X_0=X\times \mathbb{A}^1_k & X_1 \ar[l]_>>>>>{\pi _1} & \ldots \ar[l]_{\pi _2} & X_l \ar[l]_{\pi _l} & \ldots \\
	\xi _0=(\xi ,0) & \xi _1 & \ldots & \xi _l & \ldots 
}\end{equation}
Here, $\pi _i $ is the blow up of $X_{i-1}$ at $\xi _{i-1}$, where $\xi _{i}=\mathrm{Im}(\Gamma _{i})\cap \pi _{i}^{-1}(\xi _{i-1})$ for $i=1,\ldots ,l,\ldots $, and $\Gamma _i$ is the (unique) arc in $X_i$ with center $\xi _i$ which is obtained by lifting $\Gamma _0$ via the proper birational morphism $\pi _1\circ \ldots \circ \pi _i$.
This sequence of blow ups defines a non-increasing sequence
\begin{equation}
\label{Nash_sequence_2}
m_0\geq m_1\geq \ldots \geq m_l=m_{l+1}=...\geq 1, 
\end{equation}
where $m_i$ corresponds to the multiplicity of $X_i$ at $\xi _i$ for each $i=0,\ldots ,l,\ldots $. Note that $m_0$ is nothing but the multiplicity of $X$ at $\xi $, and it is proven that for hypersurfaces  the sequence (\ref{Nash_sequence_2}) coincides with the sequence (\ref{Nash_sequence}) above.  We will refer to the sequence of blow ups in (\ref{intro:diag:Nms}) as the {\em sequence of blow ups directed by $\varphi$}.
\bigskip

\noindent {\bf The  persistance}
\medskip

\noindent Let $\varphi\in \L(X, {\xi})$ be an arc   whose generic point is not contained in the stratum of multiplicity $m_0$ of $X$,  and consider, as in (\ref{Nash_sequence_2}), the Nash multiplicity sequence along $\varphi $. For the purposes of this paper, we will pay attention to the first time that the Nash multiplicity drops below $m_0$,
   see \cite[\S 2, Theorem 5]{L-J} and the discussion in \ref{setting}  below).  

\begin{Def}\label{def:rho}
	Let $\varphi $ be an arc in $X$ with center $\xi \in X$, a point of multiplicity $m_0>1$. Suppose that the 
  generic point of $\varphi$ is not contained in the stratum of points of multiplicity $m_0$ of $X$.  We denote by $\rho _{X,\varphi }$ the minimum number of blow ups directed by $\varphi $ which are needed to lower the  Nash multiplicity of  $X$ at $\xi $. That is, $\rho _{X,\varphi }$ is such that $m_0=\ldots =m_{\rho _{X,\varphi }-1}>m_{\rho _{X,\varphi }}$ in the sequence (\ref{Nash_sequence_2}) above. We call $\rho _{X,\varphi }$ the \textit{persistance of $\varphi$}. 
\end{Def}

To keep the notation as simple as possible, $\rho _{X,\varphi }$ does not contain a reference to the point $\xi $, since it is  determined by the center of $\varphi$.  

\begin{Rem}\label{Hickel_construction}
	Using Hickel's construction, it can be checked that the  first index $i\in \{1,\ldots, l+1\}$ for which there is a strict inequality   in (\ref{Nash_sequence_2})  (i.e., the first index $i$ for which  $m_0>m_i$) can be interpreted as the minimum   number of blow ups needed to {\em separate  the graph of $\varphi$ from } the stratum of points of multiplicity $m_0$ of $X_0$     (actually, to be precise, this statement has to be interpreted  in $B\otimes K[[t]]$, where the graph of $\varphi$ is defined). 
\end{Rem}

Next we  define a normalized version  of $\rho _{X,\varphi }$  in order to avoid the influence of the order of the arc in the number of blow ups needed to lower the  Nash multiplicity. 

\begin{Def}\label{def:rho_bar}  
	For a given arc $\varphi: \text{Spec}(K[[t]])\to X$   with center $\xi\in X$, we will write
	$$\bar{\rho }_{X,\varphi }=\frac{\rho _{X,\varphi }}{\nu_t(\varphi )}\mbox{, }  $$
	where $\nu_t(\varphi)$ denotes the oder of the arc, i.e., the usual order of $\varphi({\mathfrak m}_{\xi})$ at $K[[t]]$.   
\end{Def}

\begin{Def} \label{rho_funcion}  For   each point $\xi\in X$  we define the functions: 
\begin{equation}
\label{funcion_rho_1}
\begin{array}{rrclcrrcl}
\rho_X: & \L(X, \xi)& \to & {\mathbb Q}_{\geq 0} \cup \{\infty\} &    \ \ \text{ and } \ \   &  \overline{\rho}_X: & \L(X, \xi) & \to & {\mathbb Q}_{\geq 0 } \cup \{\infty\}\\ 
& \varphi & \mapsto &  \rho_{X,\varphi} & &  
& \varphi & \mapsto &  \overline{\rho}_{X,\varphi}.
\end{array} 
\end{equation}  
\end{Def}

\

\noindent{\bf  The  ${\mathbb Q}$-persistance}
\medskip

\noindent In our arguments we will be using a refinement of the persistance: the ${\mathbb Q}$-persistance.  As we will see both notions are closely related. 

\begin{Def}\label{q_persistencia}
Let $\varphi $ be an arc in $X$ with center $\xi \in X$, a point of multiplicity $m_0>1$, say $\varphi :\mathrm{Spec}(K[[t]])\longrightarrow X$. Consider the family of arcs given as $\varphi _n=\varphi \circ i_n$ for $n>1$, where $i_n^*:K[[t]]\longrightarrow K[[t^n]]$ maps $t$ to $t^n$. Then the {\em ${\mathbb Q}$-persistance of $\varphi$}, ${r}_{X,\varphi}$, is defined as the limit: 
\begin{equation}\label{r_limite_n} 
{r}_{X,\varphi }:= \lim _{n\rightarrow \infty }\frac{\rho _{X,\varphi _{n}}}{n}.  
\end{equation}
And the {\em normalized ${\mathbb Q}$-persistance of $\varphi$} is:
\begin{equation}\label{r_limite} 
\bar{r}_{X,\varphi }:=\frac{r_{X,\varphi}}{\nu_t({\varphi})}=\frac{1}{\nu_t(\varphi )}\cdot \lim _{n\rightarrow \infty }\frac{\rho _{X,\varphi _{n}}}{n}.  
\end{equation}	
\end{Def}
In \ref{setting} we will justify that both limits (\ref{r_limite_n}) and (\ref{r_limite}) exist. In fact, we will also see that the ${\mathbb Q}$-persistance of $\varphi$ can  somehow be  interpreted as the {\em order of contact of the arc $\varphi$ with the stratum of multiplicity $m_0$ of the variety $X_0$}.
 
 \begin{Def}
For   each point $\xi\in X$  we define the functions: 
\begin{equation}
\label{funcion_r}
\begin{array}{rrclcrrcl}
r_X: & \L(X, \xi)& \to & {\mathbb Q}_{\geq 0} \cup \{\infty\} &    \ \ \text{ and } \ \   &  \overline{r}_X: & \L(X, \xi) & \to & {\mathbb Q}_{\geq 0 } \cup \{\infty\}\\ 
& \varphi & \mapsto &  r_{X,\varphi} & &  
& \varphi & \mapsto &  \overline{r}_{X,\varphi}.
\end{array} 
\end{equation} 
\end{Def}

\begin{Rem}
Note that both, functions $\rho_{X}$ and $r_{X}$ are two invariants that encode the same piece of information.  
On the one hand, for each arc $\varphi$,  it can be shown   that   $\rho _{X,\varphi }$ can be  obtained by taking the integral part of $r_{X,\varphi }$ (see \cite[Proposition 5.11]{BEP2}, and also \cite{Br_E_P-E}). On the other,  expression (\ref{r_limite}) indicates that the function  $\overline{r}_{X}$ can be read from the function $\overline{\rho}_X$. 
\end{Rem}

\begin{Rem} \label{rho_etale}
{\em The persistance is stable by \'etale morphisms.}  In fact the whole sequence $\{m_i\}_{i\geq 0}$ in (\ref{Nash_sequence_2}) does not change in an \'etale neighborhood of $\xi\in X$  in the sense that we explain in the following lines. Using Remarks \ref{rem:lifting_arc_etale} and \ref{Rem:Blup_etale}, diagram (\ref{intro:diag:Nms}) can be lifted by pull back with $\mu$:
\begin{equation*}
\xymatrix@R=15pt@C=50pt{
	\Spec(K'[[t]]) \ar[dd]_{\Gamma'_0} \ar[ddr]^{\Gamma'_1} \ar[ddrrr]^{\Gamma'_l} \ar@/_5pc/[ddddd] & & & & \\
	& & & & \\
	X'_0=X'\times \mathbb{A}^1_k \ar[d]_{\mu_0} & X'_1 \ar[l]_>>>>>{\pi _1} \ar[d]_{\mu_1} & \ldots \ar[l]_{\pi _2} & X'_l \ar[l]_{\pi _l} \ar[d]_{\mu_l} & \ldots \\
	X_0=X\times \mathbb{A}^1_k & X_1 \ar[l]_>>>>>{\pi _1} & \ldots \ar[l]_{\pi _2} & X_l \ar[l]_{\pi _l} & \ldots \\
	& & & & \\
	\Spec(K[[t]]) \ar[uu]^{\Gamma _0} \ar[uur]_{\Gamma _1} \ar[uurrr]_{\Gamma _l} & & & & 
}\end{equation*}
Setting $\xi'_0=(\xi',0)$, observe that for each index $i=0,1,\ldots$,  one has that: 
\begin{enumerate}
\item[(i)] The morphism $\mu_i$ is  \'etale ; 
\item[(ii)] Each arc  $\Gamma'_i$ is a lifting of $\Gamma_i$;  
\item[(iii)] If we set  $\xi'_i$ as the center of $\Gamma'_i$, then  $\mu_i(\xi'_i)=\xi_i$. 
\end{enumerate}

Suppose that the  Nash multiplicity sequence for the arc $\varphi'$ is $\{m'_i\}_{i\geq 0}$, where $m'_i=\mult_{\xi'_i}(X'_i)$.
Since all the morphisms $\mu_i$ are \'etale  it can be concluded that  $m_i=m'_i$ for all $i\geq 0$.
In particular the persistance  of $\varphi$ is the same as the persistance  of $\varphi'$, and so  is  the normalized persistance   at  $\varphi$  and   $\varphi'$, i.e.,  
$$\rho_{X,\varphi}=\rho_{X',\varphi'} \ \text{ and } \  \overline{\rho}_{X,\varphi}=\overline{\rho}_{X',\varphi'}.$$

Finally, since any arc with center $\xi'$ induces an arc with center $\xi$, it can be concluded that we also have equality for the infimum value at $\xi$ and at $\xi'$,
$$\min\{\rho_{X,\varphi} \mid \varphi\in\mathcal{L}(X,\xi)\}=
\min\{\rho_{X',\varphi'} \mid \varphi'\in\mathcal{L}(X',\xi')\} \text{ and } $$
$$\min\{\overline{\rho}_{X,\varphi} \mid \varphi\in\mathcal{L}(X,\xi)\}=
\min\{\overline{\rho}_{X',\varphi'} \mid \varphi'\in\mathcal{L}(X',\xi')\}.$$ 
 From here it also follows that the same equalities hold for the ${\mathbb Q}$-persistance. 
\end{Rem}

\section{Rees algebras}\label{Rees_Algebras}

The stratum  defined by the  maximum value of the multiplicity function of a variety can be described using equations and weights (\cite{V}).
The same occurs with the Hilbert-Samuel function (\cite{Hir1}). As we will see, such descriptions are convenient when addressing a resolution of singularities by a composition of blow ups at suitably chosen regular centers. 
Rees algebras are natural objects to work with this setting, with the advantage that we can  perform algebraic operations on them such as taking the  integral  closure or the saturation by the action of  differential operators (the later if we work   on  smooth schemes defined over  perfect fields).   

\begin{Def}
	Let $R$ be a Noetherian ring. A \textit{Rees algebra $\mathcal{G}$ over $R$} is a finitely generated graded $R$-algebra
	$$\mathcal{G}=\bigoplus _{l\in \mathbb{N}}I_{l}W^l\subset R[W]$$
	for some ideals $I_l\in R$, $l\in \mathbb{N}$ such that $I_0=R$ and $I_lI_j\subset I_{l+j}\mbox{,\; } \forall l,j\in \mathbb{N}$. Here, $W$ is just a variable in charge of the degree of the ideals $I_l$. Since $\mathcal{G}$ is finitely generated, there exist some $f_1,\ldots ,f_r\in R$ and positive integers (weights) $n_1,\ldots ,n_r\in \mathbb{N}$ such that
	\begin{equation}\label{def:Rees_alg_generadores}
	\mathcal{G}=R[f_1W^{n_1},\ldots ,f_rW^{n_r}]\mbox{.}
	\end{equation}
\end{Def}

\begin{Rem}\label{casi_anillos}
	Note that this definition is more general than the (usual) one considering only algebras of the form $R[IW]$ for some ideal $I\subset R$, which we call Rees rings, where all generators have weight one. There is another special type of Rees algebras that will play a role in our arguments. We refer to them as {\em almost Rees rings}, and they are Rees algebras of the form $R[IW^b]$, for some ideal $I\subset R$ and some positive integer $b$ (i.e., these algebras are generated by the elements of the ideal $I$ in weight $b$). Finally,  Rees algebras can be defined over Noetherian schemes in the obvious manner.
\end{Rem}

\begin{Def}  
	Two Rees algebras over a Noetherian ring $R$  are \textit{integrally equivalent} if their integral closure in $\mathrm{Quot}(R)[W]$ coincide. We say that a Rees algebra over $R$, $\mathcal{G}=\oplus _{l\geq 0}I_lW^l$ is \textit{integrally closed} if it is integrally closed as an $R$-ring in $\mathrm{Quot}(R)[W]$. We denote by $\overline{\mathcal{G}}$ the integral closure of $\mathcal{G}$.
\end{Def}

\begin{Rem}\label{todas_casi}
	Note that $\overline{\mathcal{G}}$  is also a Rees algebra over $R$ (\cite[\S 1.1]{Br_G-E_V}).  It can be shown that any Rees algebra $\G=\oplus_lI_lW^l$ is finite over an almost Rees ring, i.e., there is  some positive integer $N$ such that $\G$ is finite over $R[I_NW^N]$    (see \cite[Remark 1.3]{E_V}). 
\end{Rem} 

\begin{Parrafo}
	\textbf{The Singular Locus of a Rees Algebra.} (\cite[Proposition 1.4]{E_V}). When working over smooth schemes one can attach to a Rees algebra a closed set as follows. 
	Let $\mathcal{G}$ be a Rees algebra over a smooth scheme $V$ defined over a perfect field $k$. The \textit{singular locus} of $\mathcal{G}$, Sing$(\mathcal{G})$, is the closed set given by all the points $\xi \in V$ such that $\nu _{\xi }(I_l)\geq l$, $\forall l\in \mathbb{N}$, where $\nu _{\xi}(I)$ denotes the order of the ideal $I$ in the regular local ring $\mathcal{O}_{V,\xi }$. If $\mathcal{G}=R[f_1W^{n_1},\ldots ,f_rW^{n_r}]$, the singular locus of $\mathcal{G}$ can be computed as
	$$\mathrm{Sing}(\mathcal{G})=\left\{ \xi \in \mathrm{Spec}(R):\, \nu _{\xi }(f_i)\geq n_i,\; \forall i=1,\ldots ,r\right\} \subset V\mbox{.}$$
\end{Parrafo}

Note that the singular locus of the ${\mathcal O}_V$-Rees algebra    generated by $f_1W^{n_1}, \ldots ,f_rW^{n_r}$ does not coincide with the usual definition of the singular locus of the subscheme of $V$ defined by $f_1,\ldots ,f_r$.

\begin{Ex} \label{Ex:HiperRees} Suppose that $R$ is smooth over a perfect field $k$. 
	Let $X\subset\Spec(R)=V$ be a hypersurface with $I(X)=(f)$ and let $b>1$ be the maximum value of the multiplicity of $X$.
If we set $\mathcal{G}=R[fW^b]$ then $\Sing(\mathcal{G})=\Mm$ is the set of points of $X$ having maximum multiplicity. Along this paper we will be using a   generalization of this description of the maximum multiplicity locus  in the case where $X$ is an   equidimensional singular algebraic variety (defined over a perfect field $k$)   (see Theorem \ref{Th:PresFinita} and the discussion in \ref{setting}). 
\end{Ex}

\begin{Parrafo}
	{\bf Singular locus, integral closure and differential saturation.}  
	A Rees algebra $\mathcal{G}=\oplus _{l\geq 0}I_lW^l$ defined on  a smooth scheme $V$ over a perfect field $k$,  is \textit{differentially closed} (or {\em differentially saturated}) if there is an affine open covering $\{U_i\}_{i\in I}$ of $V$,  such that for every $D\in \mathrm{Diff}^{r}(U_i)$ and $h\in I_l(U_i)$, we have $D(h)\in I_{l-r}(U_i)$ whenever $l\geq r$ (where $\mathrm{Diff}^{r}(U_i)$ is the locally free sheaf over $V$ of $k$-linear differential operators of order less than or equal to $r$). In particular, $I_{l+1}\subset I_l$ for $l\geq 0$. We denote by $\mathrm{Diff}(\mathcal{G})$ the smallest differential Rees algebra containing $\mathcal{G}$ (its \textit{differential closure}). (See \cite[Theorem 3.4]{V07} for the existence and construction.)
	
	It can be shown (see \cite[Proposition 4.4 (1), (3)]{V3}) that for a given Rees algebra $\mathcal{G}$  on $V$,  
	$$\Sing(\G) =\Sing (\overline{\G})= \Sing (\mathrm{Diff}(\mathcal{G})).$$
	
\end{Parrafo}
As we will see in Section \ref{Presentaciones_locales}, the problem of {\em simplification of the  multiplicity of an algebraic variety}  can  be translated into the problem of {\em resolution of a suitably defined Rees algebra} (see (\ref{simple_1}) and (\ref{simple_2})). This motivates Definitions \ref{def:transf_law} and \ref{def:res_RA} below   (see also Example \ref{ejemplo_hipersuperficie}). 

\begin{Def}\label{def:transf_law}
	Let $\mathcal{G}$ be a Rees algebra on    a smooth scheme $V$. A \textit{$\mathcal{G}$-permissible blow up}
	$$V\stackrel{\pi }{\leftarrow} V_1\mbox{,}$$
	is the blow up of $V$ at a smooth closed subset $Y\subset V$ contained in $\mathrm{Sing}(\mathcal{G})$ (a {\em permissible center for $\mathcal{G}$}). We denote then by $\mathcal{G}_1$ the (weighted) transform of $\mathcal{G}$ by $\pi $, which is defined as
	$$\mathcal{G}_1:=\bigoplus _{l\in \mathbb{N}}I_{l,1}W^l\mbox{,}$$
	where 
	\begin{equation}\label{eq:transf_law}
	I_{l,1}=I_l\mathcal{O}_{V_1}\cdot I(E)^{-l}
	\end{equation}
	for $l\in \mathbb{N}$ and $E$ the exceptional divisor of the blow up $V\stackrel{\pi }{\leftarrow} V_1$.
\end{Def}

\begin{Def}\label{def:res_RA}
	Let $\mathcal{G}$ be a Rees algebra over a smooth scheme $V$. A \textit{resolution of $\mathcal{G}$} is a finite sequence of blow ups
	\begin{equation}\label{diag:res_Rees_algebra}
	\xymatrix@R=0pt@C=30pt{
		V=V_0 & V_1 \ar[l]_>>>>>{\pi _1} & \ldots \ar[l]_{\pi _2} & V_l \ar[l]_{\pi _l}\\
		\mathcal{G}=\mathcal{G}_0 & \mathcal{G}_1 \ar[l] & \ldots \ar[l] & \mathcal{G}_l \ar[l]
	}\end{equation}
	at permissible centers $Y_i\subset \text{Sing} ({\mathcal G}_i)$, $i=0,\ldots, l-1$, such that $\mathrm{Sing}(\mathcal{G}_l)=\emptyset$, and such that the exceptional divisor of the composition $V_0\longleftarrow V_l$ is a union of hypersurfaces with normal crossings. Recall that a set of hypersurfaces $\{H_1,\ldots, H_r\}$ in a smooth $n$-dimensional $V$  has normal crossings at a point $\xi\in V$ if there is a regular system of parameters $x_1,\ldots, x_n\in {\mathcal O}_{V, \xi}$ such that if $\xi\in H_{i_1}\cap \ldots \cap H_{i_s}$,  and $\xi \notin H_l$ for $l\in \{1,\ldots, r\}\setminus \{i_1,\ldots,i_s\}$, then
	${\mathcal I}(H_{i_j})_{\xi}=\langle x_{i_j}\rangle$  for $i_j\in \{i_1,\ldots, i_s\}$;
	we say that $H_1,\ldots, H_r$ have normal crossings in V if they have normal crossings at each point of $V$.  
\end{Def}

\begin{Ex}\label{ejemplo_hipersuperficie}
	With the setting of Example~\ref{Ex:HiperRees}, a resolution of the Rees algebra $\mathcal{G}=R[fW^b]$ induces a sequence of transformations  such that  the multiplicity of the strict transform of $X$  decreases:
	\begin{gather*}
	\xymatrix@R=0pt@C=30pt{
		\mathcal{G}=\mathcal{G}_0 & \mathcal{G}_1 \ar[l] & \ldots \ar[l] & \mathcal{G}_{l-1} \ar[l] & \mathcal{G}_l \ar[l]\\
		V=V_0 & V_1 \ar[l]_>>>>>{\pi _1} & \ldots \ar[l]_{\pi _2} & V_{l-1}\ar[l]_{\pi_{l-1}} & V_l \ar[l]_{\pi _l}\\
		\  \; \; \; \; \; \cup & \cup & & \cup & \cup \\
		X=X_0 & X_1 \ar[l]_>>>>>{\pi _1} & \ldots \ar[l]_{\pi _2} & X_{l-1} \ar[l]_{\pi_{l-1}} & X_l \ar[l]_{\pi _l}
	} \\
	b=\max\mult(X_0) =  \max\mult(X_1) =\cdots = \max\mult(X_{l-1})>\max\mult(X_l).
	\end{gather*}
	Here  each $X_i$ is the strict transform of $X_{i-1}$ after the blow up $\pi_i$.
	Note that the set of points of $X_l$ having multiplicity $b$ is $\Sing(\mathcal{G}_{l})=\emptyset$.
\end{Ex}

\begin{Rem}
Resolution of Rees algebras is known to exists when $V$  is a smooth scheme defined over a field of characteristic zero  (\cite{Hir}, \cite{Hir1}). In \cite{V1} and \cite{B-M} different algorithms of resolution of Rees algebras are presented (see also \cite{E_V97}, \cite{E_Hau}). An algorithmic  resolution  requires the definition of invariants associated with the points of the singular locus of a given Rees algebra so as to define a stratification of this closed set. This is a way to select the permissible centers to blow up.  The most important invariant involved in the resolution process is {\em Hironaka's order function} defined below.  
\end{Rem}

\begin{Parrafo}\label{Def:HirOrd}
\textbf{Hironaka's order function for Rees algebras.}
(\cite[Proposition 6.4.1]{E_V})
Let $V$ be a smooth scheme over a perfect field $k$ and let $\G$ be an ${\mathcal O}_V$-Rees algebra. 
We define the \textit{order of an element $fW^{n}\in \mathcal{G}$ at $\xi \in \mathrm{Sing}(\mathcal{G})$} as
$$\mathrm{ord}_{\xi }(fW^{n}):=\frac{\nu _{\xi }(f)}{n}\mbox{.}$$
We define the \textit{order of the Rees algebra $\mathcal{G}$ at $\xi \in \mathrm{Sing}(\mathcal{G})$} as the infimum of the orders of the elements of $\mathcal{G}$ at $\xi$, that is
	$$\mathrm{ord}_{\xi }(\mathcal{G}):=\inf _{l\geq 0}\left\{ \frac{\nu_{\xi }(I_l)}{l}\right\} \mbox{.}$$
	This is what we call {\em Hironaka's order function of $\G$ at the point $\xi$}. 
	If $\mathcal{G}=R[f_1W^{n_1},\ldots ,f_rW^{n_r}]$ and $\xi\in \mathrm{Sing}(\mathcal{G})$ then it can be shown (see \cite[Proposition 6.4.1]{E_V}) that: 
	$$\mathrm{ord}_{\xi }(\mathcal{G})=\min _{i=1,\ldots,r}\left\{ \mathrm{ord}_{\xi }(f_iW^{n_i})\right\} \mbox{.}$$
	It can be proven that for any point $\xi\in\Sing(\mathcal{G})$ we have
	$\ord_{\xi}(\mathcal{G})=\ord_{\xi}(\overline{\mathcal{G}})=\ord_{\xi}(\Diff(\mathcal{G}))$  
	(see   \cite[Remark 3.5, Proposition 6.4 (2)]{E_V}).    Finally, along this paper  we use  `$\nu$'  to denote the    usual order of an element or an ideal at a regular local ring, and     `$\ord$' for  the order of a Rees algebra  at a regular local ring. 
	\end{Parrafo}

\begin{Rem}
Let  $V$ be a smooth scheme over a field of characteristic zero $k$,  and let  $\G$ be a Rees algebra on  $V$. Then  it can be shown that $\G$, $\overline{\G}$ and $\Diff \G$ share the same resolution invariants and therefore a resolution of any of them induces (naturally) a resolution of any of the others (\cite[Proposition 3.4, Theorem 4.1,  Theorem 7.18]{E_V}, \cite{Villamayor2005_2}).  
  
\end{Rem}

\section{Local presentations of the Multiplicity} \label{Presentaciones_locales}

Let $X$ be an equidimensional algebraic variety of dimension $d$   defined over a perfect field $k$.
Consider the multiplicity function
\begin{align*}
\mult_X: X & \longrightarrow \mathbb{N} \\
\xi & \longrightarrow \mult_X(\xi)=\text{mult}_{{\mathfrak m}_{\xi}} {\mathcal O}_{X,\xi}
\end{align*}
where $\text{mult}_{{\mathfrak m}_{\xi}} {\mathcal O}_{X,\xi}$ denotes   the multiplicity of the local ring $\mathcal{O}_{X,\xi}$ at the maximal ideal $\mathfrak{m}_{\xi}$. It is known that the function $\mult_X$ is upper-semi-continuous (see \cite{Dade}). In particular, suppose that $m_0$ is the maximum value of the multiplicity at points of $X$, i.e.,  suppose that $m_0=\max\mult_X$,  then the set
$$\mathrm{\underline{Max}}\mult_X:=\left\lbrace \xi \in X \mid \mult_X(\xi)\geq m_0\right\} =
\left\{ \xi \in X\mid \mult_{X}(\xi)=m_0\right\rbrace $$
is closed  (although not necessarily regular). It is also known  that the multiplicity function can not increase after a blow up $\phi:X'\to X$ with   regular   center $Y$
provided that $Y\subset \mathrm{\underline{Max}}\mult_X$ (cf. \cite{Dade}).  This means that $\mult_{X'}(\xi')\leq \mult_X(\xi)$ for  
$\xi=\phi(\xi')$, $\xi '\in X'$.

One could try to approach  a resolution of singularities by defining a sequence of blow ups   at regular equimultiple centers  
\begin{equation}
\label{simple_1}
\xymatrix{X=X_0 & X_1\ar[l] & \ar[l] \ldots & \ar[l] X_{l-1} & \ar[l] X_l}
\end{equation}
so that 
\begin{equation}
\label{simple_2}
m_0=\max\mult_{X_0}=\max\mult_{X_1} =\ldots =\max\mult_{X_{l-1}} >\max\mult_{X_l}. 
\end{equation}
A sequence like (\ref{simple_1})  with the property (\ref{simple_2})  is a {\em simplification of the multiplicity of $X$}.  
\medskip

One way to approach a simplification of the multiplicity of $X$ is by describing the set $\Mm$ via the singular locus of a suitably chosen Rees algebra $\G$, defined on some smooth scheme $V$, and then trying to find a resolution of $\G$ (compare with Examples \ref{Ex:HiperRees} and 
\ref{ejemplo_hipersuperficie} where the case of hypersurfaces is treated).  To be more precise, in \cite{V} it is proven that  for each  $\xi\in \mathrm{\underline{Max}}\mult_X$   there is an (\'etale) neighborhood  $U\subset X$ of $\xi$     which we denote again by $X$ to ease the notation,     and  an embedding 
$X\subset V=\Spec(R)$ for some smooth $k$-algebra $R$, together with  an  $R$-Rees algebra, $\mathcal{G}$, 
so that  
\begin{equation} \label{eq:LocPresSing}
\mathrm{\underline{Max}}\mult_{X}=\Sing(\mathcal{G}). 
\end{equation}
and so that, in addition, given a sequence of blow ups at regular equimultiple  centers,  
\begin{equation} \label{eq:LocalGSeq}
\xymatrix@R=0pt@C=30pt{
	V=V_0 & V_1 \ar[l]_>>>>>{\phi _1} & \ldots \ar[l]_{\phi _2} & V_l \ar[l]_{\phi _l}\\
	\  \; \; \; \; \; \cup & \cup & & \cup \\
	X=X_0 & X_1 \ar[l] & \ldots \ar[l] & X_l \ar[l] \\
	\mathcal{G}=\mathcal{G}_0 & \mathcal{G}_1 & \ldots  & \mathcal{G}_l
}\end{equation}
 the following equality of closed subsets holds:  
\begin{equation} \label{eq:LocPresSing2}
\left\lbrace \xi\in X_j \mid \mult_{X_j}(\xi)=m_0 \right\rbrace=
\Sing(\mathcal{G}_j), 
\qquad j=0,1,\ldots,l.
\end{equation}

It is worth mentioning that  in fact, the link between the maximum multiplicity locus of $X$ and the Rees algebra $\G$ is much stronger (it can be checked that equality (\ref{eq:LocPresSing2}) is also  preserved after considering smooth morphisms or restrictions to open subsets).
Thus  the problem of finding a simplification of the multiplicity of an algebraic variety is  translated into the problem of finding a resolution of a suitable Rees algebra defined  on  a smooth scheme.
And this can be done when the characteristic of the base field is zero.
The local embedding together with the Rees algebra $\G$  strongly linked to $\Mm$ is what we call a {\em local presentation of the multiplicity}, and we will use the notation $(V,\G)$.   Precise statements about local presentations can be found for instance  in \cite[Part II]{Br_V1} or in \cite{Su_V}.

\begin{Thm} \label{Th:PresFinita}
	\cite[7.1]{V}
	Let $X$ be a reduced equidimensional scheme  of finite type over a perfect field $k$.
	Then for every point $\xi\in X$ there exists a local presentation for the function $\mult_X$ in an (\'etale) neighborhood of $\xi$. 
\end{Thm}

\begin{Rem}
	Local presentations are not unique. For instance, once a local (\'etale) embedding $X\subset V$  is fixed, there may be different  ${\mathcal O}_V$-Rees algebras representing $\Mm$. However, it can be proven that they all lead to the same simplification of the multiplicity of $X$, i.e., they all lead to the same sequence (\ref{eq:LocalGSeq}) with $\Sing\G_l=\emptyset$ (at least in characteristic zero, see \cite{Br_G-E_V}, \cite{Br_V2} and \cite{E_V}). 
\end{Rem} 

\section{Hironaka's order function, the persistance, and the ${\mathbb Q}$-persistance} \label{Hironaka_qpersistance}

 For a given $d$-dimensional singular algebraic variey $X$ defined over a perfect field,  and once a local presentation of the multiplicity is chosen, say $(V,\G)$ (see Section \ref{Presentaciones_locales}), one would like to design an algorithm to find a resolution of $\G$ (i.e., an algorithm to find a simplification ot the multiplicity of $X$). When the characteristic is zero this is done via the so called {\em resolution invariants} that are used to asign a string of numbers to each point $\xi \in \Mm= \Sing (\G)$. In this way  one  can  define  an upper semi-continuos function $g:\Sing (\G) \to (\Gamma, \geq)$, where $\Gamma$ is some well ordered set, and whose maximum value determines the first center to blow up. This function is constructed so that its maximum value drops after each blow up and as a consequence a resolution of $\G$ is achieved after a finite number of steps. 
	
Now,  it turns out that the first relevant invariant, i.e., the first relevant coordinate of the  function $g$    is {\em Hironaka's order function in dimension $d$}, $\ord^{(d)}_X$. This function is  defined using the   so called {\em elimination algebra in dimension $d$}. 
We will not give the precise definition here;   instead we will  describe a way to construct it (full details and the precise definition can be found in  \cite{V07} and \cite{Br_V}).  We underline   that both  the elimination algebra, and Hironaka's order function in dimension $d$  can be defined in any characteristic (for the definition  it suffices to work over perfect fields). 

\

Thus, the main purpose of this section is to  establish the common setting and the notation that will be used in the proofs of our results in the following sections. To this end: 

\begin{itemize}
	\item[(i)]  We will sketch  the main ideas of the proof of Theorem \ref{Th:PresFinita} which will serve us to set a common context for our proofs; 
	\item[(ii)]   We will present a construction of the elimination algebra and give the definition of Hironaka's orden function in dimension $d$ (all this using the setting established in (i)); 
	\item[(iii)]  We will give an expression that leads to  the computation of    the persistance and the  ${\mathbb Q}$-persistance of a given arc   using the elimination algebra;
	\item[(iv)] Items (i), (ii) and (iii) are set in an \'etale neigborhood of a point $\xi\in X$; in \ref{morfismo_etale}, we will explain how the previous items  give  us enough  information to prove our results for arcs in  $X$. 
	
\end{itemize} 
\begin{Parrafo}  \label{setting} {\bf  The common setting for the proofs of the results in sections \ref{genericos}, \ref{irreducible_Hironaka} and \ref{valores}. }  Our statements are of local nature.  So, let us assume that $X$ is an affine algebraic variety of dimension $d$ over a perfect field $k$, and let $\xi\in \Mm$ be a point of multiplicity $m_0$.  Taking this starting point we now sketch some of the main lines in the proof of Theorem \ref{Th:PresFinita}, and then we will pursue  objectives (ii) and (iii) afterwards.  Most of the contents of these parts were developed and proved in \cite{Br_E_P-E} and \cite{BEP2}.
	
\

	\noindent {\bf Some ideas behind the proof of Theorem \ref{Th:PresFinita}}
	
	\
	
	\noindent  In \cite[\S5, \S7]{V} it is proven  that  after considering   suitably defined \'etale extensions, $k\subset k'$, and $\mu: X'\to X$ with $\mu(\xi')=\xi$,  we are in the following setting: 
	$X'=\Spec(B)$,  $\xi'\in \underline{\text{Max}} \text{ Mult}_{X'}$, 
	and there is a 
	smooth  $k'$-algebra,   $S$ with the following properties: 
	\begin{enumerate}
		\item[(i)] There is an extension $S\subset B$ which is finite, inducing a finite morphism 
		$$\beta: \text{Spec}(B)\to \text{Spec}(S);$$ 
		\item[(ii)]  If  $K(S)$ if  the field of fractions  of $S$ and     $\mathrm{Quot}(B)$ is  the  total quotient ring  of   $B$, then the rank of $\mathrm{Quot}(B)$ as $K(S)$-module equals $m_0$, i.e., the generic rank of $B$ as $S$-module equals $m_0=\text{max mult}_{X'}$. 
	\end{enumerate}
	
	Under these  assumptions, $B=S[\theta_1,\ldots, \theta_{n-d}]$,  for some $\theta_1,\ldots, \theta_{n-d}\in B$ and some $n>d$. Observe that the previous extension induces a natural embedding $X'\subset V^{(n)}:=\text{Spec}(R)$, where $R=S[x_1,\ldots, x_{n-d}]$. 
	
	Now, if  $f_i(x_i)\in K(S)[x_i]$ denotes the minimal polynomial of $\theta_i$ for $i=1,\ldots, (n-d)$, then it can be shown that in fact  $f_i\in S[x_i]$,  and as a consequence $\langle f_1(x_1), \ldots, f_{n-d}(x_{n-d})\rangle\subset {\mathcal I}(X')$,  the defining ideal of $X'$ in $V^{(n)}$.  Finally, if each polinomial $f_i$ is of degree $m_i$,  it can be  proven that the differential Rees algebra
	\begin{equation}
	\label{G_Representa}
	\Gn:=\Diff(R[f_1W^{m_1}, \ldots, f_{n-d}W^{m_{n-d}}])
	\end{equation}
	is a local presentation of $\underline{\text{Max}} \text{ mult}_X$   at $\xi$ (in \'etale topology). Therefore, a resolution of $\Gn$ induces a simplification of the multiplicity of $X$ (see (\ref{simple_1}) and (\ref{simple_2})). The pair $(\Vn, \Gn)$ gives the local presentation of the multiplicity stated in Theorem \ref{Th:PresFinita}.

	\
	
	\noindent	{\bf The elimination algebra and Hironaka's order function in dimension $d$} 
	
	\
	
	\noindent Following the previous argument, denote by $$\alpha: \Spec(S[x_1,\ldots, x_{n-d}])\to \Spec(S)$$ the natural morphism induced by the inclusion $S\subset R= S[x_1,\ldots, x_{n-d}]$.    Taking $\Gn$ as in (\ref{G_Representa}),   up to integral closure the {\em elimination algebra in dimension $d$ over $\Vd$}  is: 
	\begin{equation}
	\label{algebra_eliminacion}
	\Gd:=\Gn\cap S[W],
	\end{equation}
	(see \cite[Definition 4.10, Theorem 4.11]{V07}, and also \cite[\S  8.11]{Br_V}). 
	Then {\em Hironaka's order function in dimension $d$} is defined as: 
	\begin{equation}
	\begin{array}{rrcll}
	\ord^{(d)}_X: & \Mm  & \to & {\mathbb Q}  &  \\
	&  \zeta  &   \mapsto & \ord_{\alpha(\zeta')}\Gd & \text{ if } \mu(\zeta')=\zeta.
	\end{array} 
	\end{equation}
	It can be shown that for each point $\zeta \in \Mm$, the number $\ord^{(d)}_X(\zeta)$ does not depend on the choice of the local presentation $(\Vn,\Gn)$ nor on the choice of the finite projection to a smooth   $d$-dimensional scheme,  so far as it is generic enough (cf. \cite[Theorem 5.5]{V07}, \cite[Theorem 10.1]{Br_V} and  \cite[\S 25]{Br_V2}).

	Finally, it is worthwhile mentioning that, when the characteristic is zero, there is a strong link between the Rees algebras $\Gn$ and $\Gd$. For instance, it can be shown that $\alpha$ induces a homeomorphism  between $\Sing \Gn$ and $\Sing \Gd$, and for each regular center $Y\subset \Sing \Gn$, $\alpha(Y)\subset \Sing \Gd$ is regular too (and viceversa). Moreover,  it can be proven that finding a resolution of $\Gn$ is equivalent to finding a resolution of $\Gd$. When the characteristic is positive, the link between $\Gn$  and $\Gd$ is weaker, since   in general,   the containment $\alpha(\Sing(\Gn))\subseteq \Sing (\Gd)$ may be strict.  See \cite[Theorem 2.9, Lemma 7.1]{V00}, \cite[Corollary 2.12, \S 6]{V07},  \cite[Theorem 28.10]{Br_V2}).

	\
	
	\noindent{\bf The restriction of $\Gn$ to $X'$}

	\
	
	\noindent  Continuing with the arguments above,  let $\G_{X'}$ denote the restriction of $\Gn$ to $X'=\Spec(B)$ (where $\Gn$ is as in   (\ref{G_Representa})). It can be shown that this ${\mathcal O}_{X'}$-Rees algebra is well defined up to integral closure (i.e., it does not depend on the choice of the local presentation, see \cite[Theorem 5.3]{Abad}). Then we have the following commutative diagram together with different Rees algebras:   
\begin{equation}\label{diagrama_presentacion}
\begin{aligned}
\xymatrix@R=1mm@C=0,7pc{
	(V^{(n)},\Gn) & & & (X',\G_{X'}) \\
	R=S[x_1,\ldots ,x_{n-d}] \ar[r] &     S[x_1,\ldots ,x_{n-d}]  /\langle f_1,\ldots, f_{n-d}\rangle \ar[rr] & \ & B \\
	& \\ & \\
		S \ar[uuu]^{\alpha^*}      
	\ar[uuurrr]_{\beta^*}  
	& &  &  \\
	(V^{(d)},\Gd).
}
\end{aligned}
\end{equation} 	
	
\noindent Now, it can be proven that the following  extension of $B$-Rees algebras
\begin{equation}
\label{ex_finita_alge}
\beta^*(\Gd)\subset \G_{X'}
\end{equation}
 is finite   (see \cite[Theorem 4.11]{V07},  the discussion in \cite[3.8]{BEP2} and also \cite[4.6]{BEP2}).  In addition, by Remark \ref{todas_casi} we can assume that, up to integral closure, 
\begin{equation}
\label{salvo_entera}
\Gd=S[IW^b]
\end{equation}
for some ideal $I\subset S$ and some positive integer $b$. As a consequence, using again that  $\beta^*(\Gd)\subset \G_{X'}$  is a finite extension, we can assume that, up to integral closure, $\G_{X'}=B[(IB)W^b]$. Finally, it can be checked that ${\mathbb V}(IB)=\text{\underline{Max} mult}_{X'}$ (where $\mathbb{V}(IB)$ denotes the Zariski closed set determined by the ideal $IB$ in $X'$).
This follows from the fact that, since $\Gn=\oplus J_nW^n$ is a differential Rees algebra, $\Sing(\Gn)={\mathbb V}(J_n)$ for all $n\geq 1$, cf. \cite[Proposition 3.9]{V07}.
	
		\
	
	\noindent{\bf On the computation of the  ${\mathbb Q}$-persistance}

	\
	
\noindent For an arc $\varphi'\in \L(X', {\xi'})$ it can be shown (\cite[(5.10.2)]{BEP2} that:   	
	\begin{equation}
	\label{r_escrito}
	r_{X',\varphi'}=\mathrm{ord}_{t}(\varphi '(\G_{X'}))\in \mathbb{Q}_{\geq 1}, 
	\end{equation}
	and hence,    
	\begin{equation}
	\label{r_escrito_n}
	\bar{r}_{X',\varphi'}=\frac{\mathrm{ord}_{t}(\varphi' (\G_{X'}))}{\nu_t(\varphi')}\in \mathbb{Q}_{\geq 1}\mbox{,}
	\end{equation}
	where, if we assume that $\G_{X'}$ is locally generated by $g_1W^{b_1},\ldots , g_{n-d}W^{b_{n-d}}$
	in some affine chart   $\Spec(B)$ of  $X'$ containing the center of the arc $\varphi': B\to K[[t]]$, then $$\varphi'({\mathcal{G}}_{X'}):=K[[t]][\varphi' (g_1)W^{b_1},\ldots ,\varphi'(g_{n-d})W^{b_{n-d}}]\subset K[[t]][W].$$    Thus  in (\ref{r_escrito_n})     $\mathrm{ord}_{t}(\varphi '(\G_{X'}))$  denotes the order of the Rees algebra at the regular local ring $K[[t]]$, while $\nu_t$ denotes the usual order      at $K[[t]]$  (see \ref{Def:HirOrd}).   
	From here it can be checked that, if the generic point of the arc $\varphi'$ is not contained in $\Mmp=\Sing (\Gn)$, then   $\varphi'(\G_{X'})\subset K[[t]]$   is a non zero Rees algebra. As a consequence, $r_{X',\varphi'}$ is finite, and so is the persistance $\rho_{X',\varphi'}$.  From equality  (\ref{r_escrito}) it also follows that the limit in (\ref{r_limite_n}) exists. 
	\bigskip

\noindent{\bf Some consequences of Zariski's multiplicity formula for finite projections} 
\medskip

\noindent  Since the generic rank of the extension $S\subset B$ equals $m_0=\text{max mult}_X$, by Zariski's multiplicity formula for finite projections (cf., \cite[Chapter 8, \S 10, Theorem 24]{Z-SII}) it follows that: 
\begin{enumerate}
\item The point $\xi'$ is the unique point in the fiber over $\beta(\xi')\in \text{Spec}(S)$; 
\item The residue fields at $\xi'$ and $\beta(\xi')$ are isomorphic; 
\item The defining  ideal of $\beta(\xi')$ at $S$, $\mathfrak{m}_{\beta(\xi')}$, generates a reduction of the maximal ideal of $\xi'$, $\mathfrak{m}_{\xi'}$, at $B_{\mathfrak{m}_{\xi'}}$. 
\end{enumerate}
Observe that for a given arc $\varphi': B\to K[[t]]$  in $X$ we obtain,  by composition,   an arc $\widetilde{\varphi}': S \to K[[t]]$  in $V^{(d)}$, and it follows that: 
\begin{equation}
\label{r_formula}
\overline{r}_{X',\varphi'}=\frac{\ord_t(\varphi'(\G_{X'}))}{\nu_t(\varphi'({\mathfrak m}_{\xi'}))}=
\frac{\ord_t(\varphi'(\beta^*(\Gd)))}{\nu_t(\varphi'({\mathfrak m}_{\xi'}))}=
\frac{\ord_t(\widetilde{\varphi}'(\Gd))}{\nu_t(\widetilde{\varphi}'({\mathfrak m}_{\beta(\xi')}))},
\end{equation}
where the second equality follows from the fact that $\beta^*(\Gd)\subset \G_{X'}$ is a finite extension (see (\ref{ex_finita_alge}));   and the third because ${\mathfrak m}_{\beta(\xi')}B$ is a reduction of ${\mathfrak m}_{\xi'}$. 
\end{Parrafo}

\begin{Parrafo}

\noindent{\bf The ${\mathbb Q}$-persistance, the persistance,  and the use of  \'etale morphisms.} \label{morfismo_etale} 
Notice that  that expressions (\ref{r_escrito}) and (\ref{r_escrito_n}) are actually computed in an \'etale neighborhood of $\xi\in X$. For an \'etale morphism $X'\to X$, if $\varphi\in \L(X, {\xi})$  we use the fact   that there is always a lifting $\varphi'\in \L(X', {\xi'})$ with  $\mu_{\infty}(\varphi')=\varphi$ (see Remark \ref{rem:lifting_arc_etale}),  and by Remark \ref{rho_etale}: 
\begin{equation}
\label{r_etale}
\bar{r}_{X,\varphi}=\frac{1}{\nu_t(\varphi )}\cdot \lim _{n\rightarrow \infty }\frac{\rho _{X,\varphi _{n}}}{n}=\frac{1}{\nu_t(\varphi' )}\cdot \lim _{n\rightarrow \infty }\frac{\rho _{X',\varphi' _{n}}}{n}=\frac{\mathrm{ord}_{t}(\varphi' (\G_{X'}))}{\nu_t(\varphi')}=\bar{r}_{X',\varphi'}.
\end{equation} 

\end{Parrafo}

Finally,  as indicated above, the function $\overline{r}_{X}$ is not upper-semi-continuous in $\mathcal{L}(X, {\xi})$. However, if   two arcs $\varphi,\psi\in\mathcal{L}(X,\xi)$  have  the same order (as arcs), and if  $\varphi\in\overline{\{\psi\}}$  then one obtains the expected inequality:

\begin{Lemma} \label{r_semicont}
Let $\varphi,\psi\in\mathcal{L}(X)$ two arcs centered at $\xi$.
If $\varphi\in\overline{\{\psi\}}$ and $\nu_t(\varphi)=\nu_t(\psi)$ then
$$\overline{r}_{X,\varphi}\geq\overline{r}_{X,\psi}.$$
\end{Lemma}

\begin{proof}
By Lemma \ref{Lema_Etale_Arco} and the discussion in Remark \ref{rho_etale}, it suffices to prove the statement    after considering an \'etale extension of $X$, which we denote again by $X$ for simplicity. Thus we can assume   that we are in the same setting as  in \ref{setting}.
 Recall that the elimination algebra is, up to integral closure, $\mathcal{G}^{(d)}=S[IW^b]$ and $\mathcal{G}_X=B[(IB)W^b]$ (see \ref{salvo_entera}). By formula (\ref{r_escrito_n}), 
\begin{equation*}
\overline{r}_{X,\varphi}=\frac{\ord_t(\varphi(\G_X))}{\nu_t(\varphi({\mathfrak m}_{\xi}))}\geq
\frac{\ord_t(\psi(\G_X))}{\nu_t(\psi({\mathfrak m}_{\xi}))}=\overline{r}_{X,\psi}. 
\end{equation*}
\end{proof}

\section{Generic values in contact loci sets} \label{genericos}

Using the work developed in \cite[\S 5]{BEP2}, we start this section by giving a stronger version of the second statement in Theorem \ref{principal}. More precisely we show that  Hironaka's order function can actually be read by considering suitably chosen divisorial arcs with center $\xi$ (see Theorem \ref{arco_divisorial} below). Next, observe that,    from the way the normalized  ${\mathbb Q}$-persistance,  $\overline{r}_X$,   is   computed  (see (\ref{r_escrito}) and (\ref{r_escrito_n})) at first glance it is not obvious    that the equality of the expression   in (\ref{igualdad_divisorial}) below, holds 
generically. We address this kind of questions in   Propositions \ref{proposicion_abierto} and  \ref{corolario_abierto}. First we fix some notation and some constructions that we will be using  along this and the following section.

\begin{Rem}\label{nuevo_c}  Let $X\longleftarrow X_1$ be  the blow up at $\xi$, and let $X_1\longleftarrow\overline{X_1}$ be the normalization.
The total transform of the maximal ideal $\mathfrak{m}_{\xi}$ is locally principal at $\overline{X_1}$.
After removing a closed set of codimension at least two in $\overline{X_1}$, we can restrict to an open set $U$ such that we have a log resolution of $\mathfrak{m}_{\xi}$:
\begin{equation}
\label{def_c_i}
\mathfrak{m}_{\xi}\mathcal{O}_{U}=I(H_1)^{c_1}\cdots I(H_{\ell})^{c_{\ell}}
\end{equation}
where the hypersurfaces $H_i$ are irreducible and have only normal crossing in $U$.
Note that the integers $c_1,\ldots,c_{\ell}$ do not depend on the choice of $U$ since the complement of $U$ in $\overline{X_1}$  has codimension larger or equal than two.

Denote by $h_i\in H_i$  the generic point of $H_i$ and let $K_i$ denote the residue field of the local ring $\mathcal{O}_{\overline{X_1},h_i}$.
Set 
\begin{equation} \label{def_minc}
c=\min\{c_1,\ldots,c_{\ell}\}.
\end{equation}

Note that if $\mu: \in X'\to X$ is an \'etale neighborhood of $\xi$ with $\mu(\xi')=\xi$,  and if we consider the normalized blow up of $X'$ at $\xi'$, $X'\longleftarrow\overline{X_1'}$,  then   one may have a different number of hypersurfaces
\begin{equation*}
\mathfrak{m}_{\xi'}\mathcal{O}_{U'}=I(H'_1)^{c'_1}\cdots I(H'_{\ell'})^{c'_{\ell'}},
\end{equation*}
at a suitable open subset $U'\subset \overline{X_1'}$,  but the sets of integers  are  the same 
$\{c_1,\ldots,c_{\ell}\}=\{c'_1,\ldots,c'_{\ell'}\}$.

Moreover  suppose $X'=\Spec(B)$ is as in the setting \ref{setting},
for some ring $B$  together with a finite morphism $\beta^*: S\to B$.
Then we have that under those hypotheses one has that ${\mathfrak m}_{\beta(\xi')}B_{{\mathfrak m}_{\xi'}}$ is a reduction of ${\mathfrak m}_{\xi'}$. 

Then, 
since $\mathfrak{m}_{\beta(\xi')}B$ is a reduction of ${\mathfrak m}_{\xi'}$,  after blowing up  $\Vd$ at ${\beta(\xi')}$ and 
$X'$ at  $\xi'$, and after considering  the normalization $\overline{X'_1}$  of $X'_1$, there is a commutative diagram, 
\begin{equation}\label{diagrama_normal}
\xymatrix{
	X' \ar[d]_\beta & \ar[l]_{{\pi}} X'_1 \ar[d]_{\beta_1} &  \ar[l] \overline{X'_1} \ar[dl]^{\overline{\beta}_1} \ar@/_1pc/[ll]_{\overline{\pi}}\\
	V^{(d)} & \ar[l]_{\tilde{\pi}} V^{(d)}_1 & 
}
\end{equation}
where $\beta_1$  is a finite morphism and so is $\overline{\beta}_1$ (see \cite[Theorem 4.4]{COA}). 

If $v_0$ is the valuation on $V^{(d)}$ defined by the maximal ideal $\mathfrak{m}_{\beta(\xi')}\subset S$,
note that the valuation ring of $v_0$ is $\mathcal{O}_{V_1^{(d)},e}$, where $E$ denotes the exceptional divisor of $\tilde{\pi}$ and $e$ is the generic  point of  $E$. 

Denote by $h'_i\in H'_i$  the generic point of $H'_i$. 
The local rings $\mathcal{O}_{\overline{X'_1},h'_i}$ correspond to valuations $v_i$, $i=1,\ldots,\ell'$, and $v_1,\ldots,v_{\ell'}$ are exactly the extensions of $v_0$ to $\overline{X'_1}$.
We will denote by $K'_i$   the residue field of $\mathcal{O}_{\overline{X'_1},h'_i}$.
Consider for any $i\in \{1,\ldots, {\ell}'\}$,  the natural morphism,
\begin{equation}
\label{ejemplo_divisorial}
\eta'_i: B \to {\mathcal O}_{\overline{X'_1}, h'_i}\to \widehat{{\mathcal O}_{\overline{X'_1}, h'_i}}\simeq K'_i[[t]].
\end{equation}
Note that $\eta'_i$ is a divisorial arc in $X'$.

Consider the $K'_i$-morphism:  $i_n: K'_i[[t]]\to K'_i[[t]]$ where $t\mapsto t^{n}$.
We will denote by $\eta'_{i,n}$ the arc obtained from $\eta'_{i}$ by composing with $i_n$
\begin{equation} \label{arco_potencia}
\eta'_{i,n}: B\stackrel{\eta'_i}{\longrightarrow} K'_i[[t]] \stackrel{i_n}{\longrightarrow} K'_i[[t]].
\end{equation}
\end{Rem}

Now we revisite Theorem~\ref{principal} and restate  the second part of that result     in  Theorem \ref{arco_divisorial}. Our purpose is to  prove  a stronger statement by showing that the arc giving the equality in (\ref{igualdad_divisorial_i})   can be chosen to be divisorial.
Compared to the proof  given in \cite{BEP2}  (see Remark~\ref{Rem_Diferencia}) here  we follow a sligthely different strategy  by considering normalized blowing ups and the commutative diagram (\ref{diagrama_normal}).
This allows us to find the desired divisorial arc.  As we indicate in the proof below, the fact that  inequatlity (\ref{desigualdad_i}) holds for all arcs also shows in our way to prove Theorem  \ref{arco_divisorial}.

\begin{Thm} \label{arco_divisorial} Let $X$ be a $d$-dimensional   algebraic  variety defined over a perfect field $k$, and let $\xi\in \Mm$. Then 
 there is a divisorial arc $\eta\in \L(X, {\xi})$ such that 
\begin{equation}
\label{igualdad_divisorial}
\ord_X^{(d)}(\xi)=\overline{r}_{X,\eta}=\frac{1}{\nu_t(\eta)}\lim_{n\to \infty}\frac{\rho_{X,\eta_{n}}}{n}.
\end{equation}

\end{Thm}	

\begin{proof} 
	
We will first prove that the theorem holds for arcs defined   in some \'etale neighborhood of  $\xi\in \Mm$, say $\mu: X'\to X$, with $\mu(\xi')=\xi$,  and after we will show that the same statement actually holds for arcs in $X$. 
	
Since the statements are of local nature, we will start by assuming  that, locally, in an \'etale  neighborhood of $\xi\in \Mm$,  $\mu: X'\to X$, with $\mu(\xi')=\xi$,    one has that $X'=\text{Spec}(B)$ for some ring $B$  together with a finite morphism $\beta^*: S\to B$ as in the setting of  \ref{setting}. Also, recall that under those hypotheses   one has that ${\mathfrak m}_{\beta(\xi')}B_{{\mathfrak m}_{\xi'}}$ is a reduction of ${\mathfrak m}_{\xi'}$. 

Now we use the construction and the notation introduced in Remark \ref{nuevo_c}. Thus  have the numbers $c_1,\ldots,c_{\ell}$ and $c$ defined   in (\ref{def_c_i})  and (\ref{def_minc}),  and  the diagram (\ref{diagrama_normal}).

Suppose that, up to integral closure, $\mathcal{G}^{(d)}=S[IW^b]$ for some ideal $I\subset S$ (see (\ref{salvo_entera})).
Then, if $\ord_{\beta(\xi')}\mathcal{G}^{(d)}=\frac{a}{b}$, one has that $\nu_{\beta(\xi')}(I)=a$. Hence, the total transform of $I$ in $\Vd_1$ is $I{\mathcal O}_{V_1^{(d)}}=\mathcal{I}({E})^aJ$, for some sheaf of ideals $J\nsubseteq {\mathcal I}({E})$. 

Observe that any arc $\varphi'\in \L(X', {\xi'})$ induces an arc  $\tilde{\varphi}'\in \L(\Vd,{\beta(\xi')})$, i.e., $\beta_{\infty}(\varphi')=\tilde{\varphi}'$,  and if $\tilde{\varphi}'$ is not constant (i.e., if $\varphi'$ is not constant) then it can be lifted to an arc   in $\tilde{\varphi}_1'\in\L(\Vd_1)$. By (\ref{r_formula}) it follows that 
\begin{equation}
\label{desigualdad_etale}
\overline{r}_{X,\varphi'}=\frac{\ord_t(\tilde{\varphi}'(\Gd))}{\nu_t(\tilde{\varphi}')}=
\frac{\frac{\nu_t(\tilde{\varphi}_1'(\mathcal{I}({E})^aJ))}{b}}{\nu_t(\tilde{\varphi}_1'(\mathcal{I}({E})))}=
\frac{a\nu_t(\tilde{\varphi}_1'(\mathcal{I}({E})))+\nu_t(\tilde{\varphi}_1'(J))}{b\nu_t(\tilde{\varphi}_1'(\mathcal{I}({E})))} \geq \frac{a}{b},
\end{equation}
which gives a proof of the inequality in (\ref{desigualdad_i}).

Now, consider for any $i\in \{1,\ldots, {\ell}'\}$,  the
divisorial arc $\eta'_i$ from  (\ref{ejemplo_divisorial}).
It can be checked that $\nu_t({\eta'_i}(IB))=ac'_i$.
Then we have 
\begin{equation}
\label{igualdad_etale}
\overline{r}_{X',\eta'_i}= \frac{\ord_t({\eta'_i}(\G_{X'}))}{\nu_t({\eta'_i})}=
\frac{\nu_t({\eta'_i}(IB))}{b\nu_t({\eta'_i})}=\frac{ac'_i}{bc'_i}=\frac{a}{b}.  
\end{equation}
To conclude, both (\ref{desigualdad_etale})  and (\ref{igualdad_etale}) are  actually proven for arcs defined in an \'etale neighborhood $X'$ of $\xi\in X$. The fact that inequality  (\ref{desigualdad_etale}) holds for arcs in $X$ follows from Remark \ref{rho_etale}. On the other hand, it can also be checked that equality (\ref{igualdad_etale}) holds for a divisorial arc in $X$: set $\overline{\mu}:= \mu \circ \overline{\pi} $,
and  observe that the divisorial arc $\eta_i'$ from (\ref{ejemplo_divisorial}) induces a commutative diagram:
\begin{equation}
\label{Divisorial_etale} 
\xymatrix{\Spec(K'_i[[t]]) \ar[rr]^{\eta_i'} \ar[d] &  &  \overline{X'_1}\ar[d]^{\overline{\mu}}\\
 \Spec(K_i[[t]]) \ar[rr]_{\eta}   & & X}
\end{equation} 
where $\eta=\overline{\mu}_{\infty}(\eta_i')$. Since $\overline{\mu}: \overline{X}_1'\to X$ is a dominant morphism between varieties of the same dimension one has that  $\eta$ is divisorial if and only if $\eta_i'$ is divisorial (see  \cite[Proposition 2.10, Lemma 3.2]{I_Crelle}, where the case of varieties over ${\mathbb C}$ is treated, and the general case follows  using  \cite[\S 6, Corollary 1, \S 14, Theorem 31]{Z-SII}).
\end{proof}	

\begin{Rem} \label{Rem_Diferencia} In the following lines we give a few indications on how the proof of 
  Theorem~\ref{principal} was addressed in \cite[\S 6.3]{BEP2}.
With the same  notation as in the proof of Theorem~\ref{arco_divisorial},
in \cite{BEP2} we only worked with the finite map $S\to B$. Then we showed that for any arc $\varphi'\in \L(X', {\xi'})$, inducing an arc  $\tilde{\varphi}'\in \L(\Vd,{\beta(\xi')})$, we obtained the inequality (\ref{desigualdad_i}) by using (\ref{r_formula}) and the properties of the order function in $S$.
This way, finding an arc giving the equality (\ref{igualdad_divisorial_i}), required more work:
\begin{itemize}
\item First, given an element $gW^b\in\mathcal{G}^{(d)}$ such that $\nu_{\beta(\xi')}(g)=a$, we needed to find an arc $\phi\in\L(\Vd,\beta(\xi'))$  such that $\nu_t(\phi(g))=a\nu_t(\mathfrak{m}_{\beta(\xi')})$. To this end we worked at the graded ring of the local ring $S_{{\mathfrak m}_{\beta(\xi')}}$ and at its  completion.
Here an \'etale extension of the base field may have to be considered.

\item   In principle we did not know much about the arc $\phi$, except that it defined a semi-valuation
(a valuation on a closed subvariety of $\Vd$) dominated by a finite number of semi-valuations on $X'$.

\item Finally it was shown that any of those semi-valuations gave us arcs fulfilling equality (\ref{igualdad_divisorial_i}). 
\end{itemize}
The key point in the proof of Theorem~\ref{arco_divisorial} is the use of the commutative diagram~(\ref{diagrama_normal}).
Given a finite morphism $S\to B$ such a commutative diagram only exists under very special conditions.
Normalized blowing ups were not considered in  \cite{BEP2}.
\end{Rem}

 
\begin{Prop} \label{proposicion_abierto}
Let $X$ be a $d$-dimensional algebraic   variety defined over a perfect field $k$,   and let $\xi\in \Mm$. 
Suppose there is some $s\geq 1$  and an arc $\varphi_0\in \text{Cont}^{=s}({\mathfrak m}_{\xi})$ with $\overline{r}_{X,{\varphi_0}}=\ord_{\xi}^{(d)}(X)$.
Then there is a non-empty open subset ${\mathfrak W}$ of $\text{Cont}^{ \geq s}({\mathfrak m}_B)$,
containing $\varphi_0$,
such that for all arcs $\varphi \in {\mathfrak W}$,  $\overline{r}_{X,\varphi}=\ord_{\xi}^{(d)}(X)$.

If, in addition,  the generic point of $\varphi_0$   is not contained in $\Sing (X)$  and the characteristic of $k$ is zero, then  there are fat (divisorial) arcs in ${\mathfrak W}$.
\end{Prop}

\begin{proof}
The statement is local, so we can assume that $X$ is an affine algebraic variety over $k$.
First we prove that the theorem holds in a suitably chosen  \'etale neighborhood of $X$, $\mu: X'\to X$ with $\mu(\xi')=\xi$.
Thus,  we will set  $X'=\Spec(B)$, and  we will be considering the finite morphism $\beta: X'\to \Vd$ with $\Vd=\Spec(S)$ a smooth $k'$-algebra as in \ref{setting}.

Now, suppose that, up to integral closure, $\mathcal{G}^{(d)}=S[IW^b]$   for some ideal $I\subset S$ (see (\ref{salvo_entera})).
Then, if $\ord_{\beta(\xi')}\mathcal{G}^{(d)}=\frac{a}{b}$, one has that $\nu_{\beta(\xi')}(I)=a$.

In $\mathcal{L}(X',\xi')$, set 
\begin{equation*}
\mathfrak{W}'={\Cont}^{\geq s}({\mathfrak m}_{\xi'})\setminus {\Cont}^{\geq as+1}(IB).
\end{equation*}
We claim that if $\varphi'\in {\mathfrak W}'$, then $$\overline{r}_{X',\varphi'}=\frac{a}{b}=\ord^{(d)}_X(\xi).$$ 
Indeed, since  $\varphi'\in {\mathfrak W}'$, we have that $\nu_t(\varphi'({\mathfrak m}_{\xi'}))=h\geq s$. If  $\beta_{\infty}(\varphi')=\tilde{\varphi}'$ denotes the   arc induced on   $\Vd$ by $\varphi'$, one has that $\nu_t(\tilde{\varphi}'({\mathfrak m}_{\beta(\xi')}))=h\geq s$, which implies that 
$$sa\leq ha\leq \nu_t(\tilde{\varphi}'(I))=\nu_t(\varphi'(IB))<sa+1.$$
Thus, necessarilly,   $s=h= \nu_t(\tilde{\varphi}'({\mathfrak m}_{\beta(\xi')}))=\nu_t(\varphi'({\mathfrak m}_{\xi'}))$, and 
\begin{equation}
\label{igualdad_abierto_etale}
\overline{r}_{X',\varphi'}=\frac{\ord_t(\varphi'(\G_{X'}))}{\nu_t(\varphi'({\mathfrak m}_{\xi'}))}= \frac{\frac{\nu_t(\varphi'(IB))}{b}}{s}=\frac{sa}{sb}=\frac{a}{b}.
\end{equation}
To finish, the proof above shows that, after considering an \'etale morphism $\mu: X'\to X$ with $\mu(\xi')=\xi$,  there is an  open subset ${\mathfrak W}'$  of  $ \Cont^{\geq s}(\mathfrak{m}_{\xi'})$   where the equality (\ref{igualdad_abierto_etale}) holds. Now, by Remark \ref{rem:lifting_arc_etale_open},
the morphism $\mu_{\infty}$ is open, and 
 $\mu_{\infty}({\mathfrak W}')= {\mathfrak W} \subseteq \Cont^{\geq s}(\mathfrak{m}_\xi)$ is an open subset of $\Cont^{\geq s}(\mathfrak{m}_\xi)$ where the statement of the theorem  holds (see Remark \ref{rho_etale}).

The last statement of the Proposition follows from Theorem \ref{ThExistDivFat},
which we postpone to section~\ref{valores}. 
\end{proof}

\begin{Prop} \label{corolario_abierto}
Let $X$ be a $d$-dimensional algebraic   variety defined over a perfect field $k$ and  let $\xi\in \Mm$.
Then for every  $n\geq 1$  and every $c_i$, $i=1,\ldots, \ell$,  there is a non-empty open set 
${\mathfrak U}_{nc_i}\subseteq\Cont^{\geq nc_i}(\mathfrak{m}_\xi)$  such that for all
$\varphi \in {\mathfrak U}_{nc_i}$, $\overline{r}_{X,\varphi}=\ord_{X}^{(d)}(\xi)$.
\end{Prop}

\begin{proof} After the proof of Proposition \ref{proposicion_abierto}, it suffices to prove the statement at an \'etale neighborhood of $\xi$.   We use the notation and the construction of 
  Remark \ref{nuevo_c}:   let $\overline{\pi}: \overline{X'_1}\to X'$ be the normalized blow up of $X'$ at $\xi'$, which induces the  commutative diagram (\ref{diagrama_normal}) of blow ups and finite morphisms. 
	
As in the proof of Theorem \ref{arco_divisorial},  we denote by  ${E}$  the exceptional divisor of $\tilde{\pi}$.  
Then ${\mathfrak m}_{\beta(\xi')}{\mathcal O}_{V^{(d)}_1}={\mathcal I}({E})$, and
$${\mathfrak m}_{\xi'}{\mathcal O}_{\overline{X'_1}}={\mathcal I}({E}){\mathcal O}_{\overline{X'_1}}. $$

After blowing up at ${\mathfrak m}_{\beta(\xi')}$,  $I{\mathcal O}_{\Vd_1}={\mathcal I}(E)^aJ$, for some sheaf of ideals $J\nsubseteq {\mathcal I}(E)$, and therefore, $\beta^*(I) {\mathcal O}_{\overline{X'_1}}={\mathcal I}(E)^aJ'$.   
Set 
$$\mathfrak{U}'_{nc'_i}:=\Cont^{\geq nc'_i}(\mathfrak{m}_{\xi'})\setminus \Cont^{\geq nc'_ia+1}(IB),$$ 
Observe that ${\mathfrak U}'_{nc'_i}$ is non-empty since  the arc $\eta'_{i,n}$  
from (\ref{arco_potencia})
belongs to  ${\mathfrak U}'_{nc'_i}$. 
\end{proof}

\section{Fat irreducible components of contact loci and Hironaka's order} \label{irreducible_Hironaka}

In the previous section we proved that given a $d$-dimensional  algebraic variety $X$ defined over a perfect field $k$, and a point of maximum multiplicity $\xi\in \Mm$,
there are locally open sets in $\L(X, {\xi})$ where the value of the normalized ${\mathbb Q}$-persistance,
$\overline{r}_{X}$, is constant and equal to the value of  
Hironaka's order at the point $\xi$, $\ord^{(d)}_{X}(\xi)$ (Proposition \ref{proposicion_abierto}).
In fact such open subsets exist for some contact sets $\Cont^{\geq nc_i}(\mathfrak{m}_{\xi})$ (see Propostion  \ref{corolario_abierto}).

In this section we will prove that the value $\ord^{(d)}_{X}(\xi)$   can be read by means of the    ${\mathbb Q}$-persistance  of some of  the irreducible (fat) components of $\Cont^{\geq s}(\mathfrak{m}_{\xi})$
for some values of $s$ (see Theorem \ref{componentes}).
It is natural to ask whether a similar statement holds for the irreducible components of  
 $\Cont^{\geq m}(\mathfrak{m}_{\xi})$  for any $m\in {\mathbb N}$, but   Example \ref{ejemplo_no} already  illustrates that this is not the case. However  we will show that  the value of Hironaka's order function at $\xi$ is obtained asymptotically by  looking at  the irreducible components of $\Cont^{\geq m}(\mathfrak{m}_{\xi})$ when  $m$ goes to infinity. This is the content of Theorem \ref{componentes_limite}.

\begin{Thm} \label{componentes}
Let $X$ be a $d$-dimensional  algebraic  variety defined over a perfect field $k$, let $\xi\in \Mm$, and let  $\{T_{\lambda_m}\}_{\lambda_m\in \Lambda_m}$  
be the fat irreducible components of   $\text{Cont}^{\geq m}({\mathfrak m}_{\xi})$, 
with generic points  $\{\Psi_{\lambda_m}\}_{\lambda_m\in \Lambda_m}$ for $m\geq 1$.   If $m=nc_i$ for some $n\geq 1$ and some $c_i$ as in (\ref{def_c_i})  then 
$$\ord_X^{(d)}(\xi) =\min \{\overline{r}_{X,\Psi_{{\lambda}_m}}: \lambda_m\in \Lambda_m\}.$$ 
In addition, if $k=\mathbb{C}$ then the minimum is achieved at the generic point of a maximal divisorial set.

\end{Thm}

\begin{proof} 
The statement is local, so we can assume that $X$ is an affine algebraic variety over $k$. First we chose a suitable \'etale neighborhood of $X$, $\mu: X'\to X$ with $\mu(\xi')=\xi$,  so that setting $X'=\Spec(B)$,  there is a finite morphism $\beta: X'\to \Vd$ with $\Vd=\Spec(S)$ a smooth $k'$-algebra in the same situation as in \ref{setting}.  Now we use the same notation and constructions as in Remark \ref{nuevo_c}. So let $\overline{\pi}: \overline{X'_1}\to X'$ be the normalized blow up of $X'$ at $\xi'$, which induces a commutative diagram of blow ups at finite morphisms as in (\ref{diagrama_normal}).

For a given $i\in \{1,\ldots, \ell\}$, we may assume after reordering that $c'_i=c_i$.
Now consider the arcs $\eta'_{i,n}$ as in (\ref{arco_potencia}).  
After the proof of Proposition \ref{corolario_abierto}, $\overline{r}_{X',\eta'_{i,n}}=\ord_X^{(d)}(\xi)$.
 
 Now define $\eta_{i,n}\in\mathcal{L}(X, {\xi})$ as  the arc obtained composing $\eta'_{i,n}$ with the \'etale morphism $\mu: X'\to X$, i.e., $\mu_{\infty}(\eta'_{i,n})=\eta_{i,n}$.
Since $\eta_{i,n}\in\Cont^{\geq nc_i}(\mathfrak{m}_{\xi})$ is fat then it belongs to some fat irreducible component of
$\Cont^{\geq nc_i}(\mathfrak{m}_{\xi})$, which we denote by  $ T_{\lambda}$,  for some $\lambda\in \Lambda_{nc_i}$, with generic point  $\Psi_{\lambda}$. Then 
notice that $\nu_t(\eta_{i,n})\geq \nu_t(\Psi_{\lambda})$, but in   fact these two numbers   are equal: by  Lemma \ref{Lema_Etale_Arco}, there is an arc $\Psi'\in\Cont^{\geq nc_i}(\mathfrak{m}_{\xi'})$ such that
$\mu_{\infty}(\Psi')=\Psi_{\lambda}$ and so that $\eta'_{i,n}\in \overline{\{\Psi'\}}$. Then 
$$nc_i=\nu_t(\eta'_{i,n})\geq \nu_t(\Psi')\geq nc_i.$$
And  the claim follows because $\mu: X'\to X$ is \'etale, and hence $\nu_t(\Psi_{\lambda})=\nu_t(\Psi')(=\nu_t(\eta'_{i,n})=\nu_t(\eta_{i,n})$). 

Since $\eta_{i,n}\in \overline{\{\Psi\}}$, by Lemma \ref{r_semicont} we have
\begin{equation*}
\frac{a}{b}=\overline{r}_{X,\eta_{i,n}}\geq  \overline{r}_{X,\Psi_{\lambda}}\geq\frac{a}{b},
\end{equation*}
where last inequality follows from Theorem \ref{arco_divisorial}. 

Finally, for $k=\mathbb{C}$ the last statement of the theorem follows from \cite[Proposition 2.12]{deF_E_I} 
(see also Proposition \ref{maximal_divisorial} in this paper).
\end{proof}

 The following example shows that the result in Theorem \ref{componentes} may not hold for the irreducible components of $\Cont^{\geq n}({\mathfrak m}_{\xi})$ for arbitrary values of  $n$. 

\begin{Ex} \label{ejemplo_no}  
Let $X$ be the hypersurface of $\mathbb{A}^3_{k}$ given by $x^2y^3-z^6=0$, with $k$ a field of characteristic zero.
Consider the contact sets $\Cont^{\geq n}({\mathfrak m}_{\xi})$ where ${\mathfrak m}_{\xi}$ is the maximal ideal of $X$ at the point $\xi =0$. For any $n\geq 11$ such that $2\nmid n$ and $3\nmid n$, any fat irreducible component of $\Cont^{\geq n}({\mathfrak m}_{\xi})$ gives a $\mathbb{Q}$-persistance strictly greater than $\ord_X^{(2)}(\xi)=1$ at its generic point. This can be computed using (\ref{UnionCont}), via the blow up of $X$ at $\xi$,
$\Pi:X_1\to X$, which gives a log-resolution of ${\mathfrak m}_{\xi}$.
We have that $\mathfrak{m}_{\xi }\mathcal{O}_{{X_1}}=I(H_1)^{2}\cdot I(H_2)^{3}$,
where $H_1$, $H_2$ are the irreducible components of excepcional divisor, according to the notation in the previous theorem.

\vspace{0.2cm}

To see this, denote
$$C_{\alpha,\beta}=
\overline{\Pi_{\infty}\left(\Cont^{(\alpha,\beta)}\left(E\right)\right)}\subset \mathcal{L}(X,\xi)$$
where the $\overline{\{\ \}}$ denotes the Zariski closure.
Let $\Psi_{\alpha,\beta}$ be the generic point of $C_{\alpha,\beta}$.

Note that the Rees algebra $\mathcal{G}_X\subset B[W]$ is generated by $xW$, $yW$ and $z^6W^5$,
  up to integral closure. 
By (\ref{UnionCont}), we have that $C_{\alpha,\beta}\subset\Cont^{\geq n}({\mathfrak m}_{\xi})$ if and only if $2\alpha+3\beta\geq n$. If $v_{\Psi_{\alpha,\beta}}$ is the valuation associated to the arc $\Psi_{\alpha,\beta}$ then it can be checked that
$v_{\Psi_{\alpha,\beta}}(x)=3\alpha+3\beta$, $v_{\Psi_{\alpha,\beta}}(y)=2\alpha+4\beta$ and $v_{\Psi }(z)=2\alpha+3\beta$
and the ${\mathbb Q}$-persistance is
$$\overline{r}_{X,\Psi_{\alpha,\beta}}=
\frac{\min\{3\alpha+3\beta,2\alpha+4\beta,\frac{6}{5}(2\alpha+3\beta)\}}{\min\{3\alpha+3\beta,2\alpha+4\beta,2\alpha+3\beta\} }=
\frac{\min\{3\alpha+3\beta,2\alpha+4\beta,\frac{6}{5}(2\alpha+3\beta)\}}{2\alpha+3\beta}.
$$
For any $(\alpha,\beta)$, it follows that $\overline{r}_{X,\Psi_{\alpha,\beta}}>1$ if and only if $\alpha\neq 0$ and $\beta\neq 0$.

Now let $n\geq 11$
be not divisible by $2$ neither by $3$. We want to prove that if $C_{\alpha,\beta}$ is an irreducible component of $\Cont^{\geq n}({\mathfrak m}_{\xi})$ then $\overline{r}_{X,\Psi_{\alpha,\beta}}>1$ (equivalently $\alpha\neq 0$ and $\beta\neq 0$).
Since $X$ is a toric variety, by \cite[Lemma 3.11]{I08} we have that
\begin{equation} \label{EqTorica}
C_{\alpha,\beta}\subset C_{\alpha',\beta'}
\Longleftrightarrow
v_{\Psi_{\alpha,\beta}}\geq v_{\Psi_{\alpha',\beta'}}.
\end{equation}
Assume that $n=2m+1=3l+i$ where $i=1$ or $2$, since $n\geq 11$ we have $m\geq 5$, $l\geq 3$.
Let $C_{\alpha,\beta}$ be an irreducible component of $\Cont^{\geq n}(\mathfrak{m}_{\xi})$ with
$\overline{r}_{X,\Psi_{\alpha,\beta}}=1$. Then either $(\alpha,\beta)=(m+1,0)$ or $(\alpha,\beta)=(0,l+1)$.

If $(\alpha,\beta)=(m+1,0)$ then $C_{m+1,0}\subset C_{m-1,1}$ by (\ref{EqTorica}),
since $v_{\Psi_{m+1,0}}(x)=3m+3\geq v_{\Psi_{m-1,1}}(x)=3m$,
$v_{\Psi_{m+1,0}}(y)=2m+2\geq v_{\Psi_{m-1,1}}(y)=2m+2$ and
$v_{\Psi_{m+1,0}}(z)=2m+2\geq v_{\Psi_{m-1,1}}(z)=2m+1$. In this case $\overline{r}_{X,\Psi_{m-1,1}}=1+\frac{1}{n}$.

If $(\alpha,\beta)=(0,l+1)$ then $C_{0,l+1,}\subset C_{2,l-1}$ in case $n=3l+1$
and $C_{0,l+1,}\subset C_{1,l}$ in case $n=3l+2$.

The first case, $n=3l+1$, comes from the inequalities:
$$\begin{array}{l}
v_{\Psi_{0,l+1}}(x)=3l+3\geq v_{\Psi_{2,l-1}}(x)=3l+3, \\
v_{\Psi_{0,l+1}}(y)=4l+4\geq v_{\Psi_{2,l-1}}(y)=4l, \\
v_{\Psi_{0,l+1}}(z)=3l+3\geq v_{\Psi_{2,l-1}}(z)=3l+1.
\end{array}
$$
Here we have that $\overline{r}_{X,\Psi_{2,l-1}}=1+\frac{2}{n}$.

The second case, $n=3l+2$, comes from the inequalities:
$$\begin{array}{l}
v_{\Psi_{0,l+1}}(x)=3l+3\geq v_{\Psi_{1,l}}(x)=3l+3, \\
v_{\Psi_{0,l+1}}(y)=4l+4\geq v_{\Psi_{1,l}}(y)=4l+2, \\
v_{\Psi_{0,l+1}}(z)=3l+3\geq v_{\Psi_{1,l}}(z)=3l+2.
\end{array}$$
And we have that $\overline{r}_{X,\Psi_{1,l}}=1+\frac{1}{n}$.
\end{Ex}

The previous computations show that if $\{T_{n,\lambda_n}\}_{\lambda_n\in\Lambda_n}$ are the irreducible components of the  set $\Cont^{\geq n}(\mathfrak{m}_{\xi})$ and $\Psi_{n,\lambda_n}$ is the generic point of $T_{n,\lambda_n}$ then,
for $n\geq 11$ with $2\nmid n$ and $3\nmid n$ we have that
$$\delta_n=\min\{\bar{r}_{\Psi_{n,\lambda_n}} \mid \lambda_n\in\Lambda_n\}>1.$$
Hovewer, note that $\lim_{n\to\infty}\delta_n=1=\ord^{(2)}_X(\xi)$, and this is a general fact as it is stated in the following theorem.

\begin{Thm} \label{componentes_limite}
Let $X$ be a $d$-dimensional algebraic variety defined over a perfect field $k$, and   let $\xi\in \Mm$.
For each $m\in {\mathbb N}$, let $\{T_{m,\lambda_m}\}_{\lambda_m\in \Lambda_m}$ be the fat irreducible components of $\Cont^{\geq m}({\mathfrak m}_{\xi})$ and let 
	$\Psi_{m, \lambda_m}$ be the generic point of $T_{m,\lambda_m}$ for    $\lambda_m\in \Lambda_m$.
	For each $m\geq 1$ set:
	\begin{equation*}
	\delta_{m}:=\inf\left\{\bar{r}_{\Psi_{m,\lambda_m}}\mid \lambda_{m} \in \Lambda_m \right\}.
	\end{equation*}
	Then we have that
	\begin{equation*}
	\ord^{(d)}_{X}(\xi)=\lim_{m\to\infty} \delta_m.
	\end{equation*}
\end{Thm}

\begin{proof}  The statement is local, so we can assume that $X$ is an affine algebraic variety over $k$.  Choose a suitable  \'etale neighborhood of $X$, $\mu: X'\to X$ with $\mu(\xi')=\xi$ so that  setting  $X'=\Spec(B)$ we are in the same situation as the one considered in \ref{setting}. Thus we will be considering the finite morphism $\beta: X'\to \Vd$ with $\Vd=\Spec(S)$ a smooth $k'$-algebra as in \ref{setting}.   Let $\overline{\pi}: \overline{X'_1}\to X'$ be the normalized blow up of $X'$ at $\xi'$, which induces a commutative diagram of blow ups at finite morphisms  as in (\ref{diagrama_normal}). We use  the same notation as in Remark \ref{nuevo_c}.
Recall that we use  $c$ for  the minimum of the set $\{c_1,\ldots, c_{\ell}\}$ and assume $c=c_1$. 

As in (\ref{salvo_entera}) assume that the Rees algebra $\mathcal{G}^{(d)}$ has the same integral closure as $\mathcal{O}_{V^{(d)}}[IW^b]$,
and assume $\ord_{\beta(\xi')}\mathcal{G}^{(d)}=a/b$.

Let $\eta'_1\in\mathcal{L}(X', {\xi'})$ be the arc defined in (\ref{ejemplo_divisorial}) for $i=1$,  and set: 
\begin{equation*}
\omega_m= \left\lceil\frac{m}{c_1}\right\rceil \qquad \text{and} \qquad
\varphi'_m:=\eta'_{1,\omega_m}, 
\end{equation*}
where $\eta'_{1,\omega_m}$ is as in (\ref{arco_potencia}). 
Let  $\eta_1\in\mathcal{L}(X,\xi)$ be the arc obtained by composing with $X'\to X$, i.e., $\mu_{\infty}(\eta'_1)=\eta_1$, and similarly, define $\varphi_m:=\mu_{\infty}(\varphi'_m)$. 

Note that
\begin{equation*}
\bar{r}_{X,\varphi_m}= \bar{r}_{X',\varphi'_m}  = \frac{\ord_t(\varphi'_m(\G_{X'}))}{\nu_t(\varphi'_m)}=   \frac{\nu_t(\varphi'_m(IB))}{b\cdot \nu_t(\varphi'_m)}=
\frac{c_1\cdot a\cdot \omega_{m}}{b\cdot c_1\cdot \omega_{m}}=\frac{a}{b}=
\ord_{\xi}\mathcal{G}^{(d)}.
\end{equation*}
By construction $\varphi_m\in\Cont^{\geq m}(\mathfrak{m})$.
Let $\lambda_m\in \Lambda_m$ be an index such that $\varphi_m\in T_{m, \lambda_m}$ with generic point $\Psi_{m, \lambda_m}$. Using Lemma \ref{Lema_Etale_Arco}, let $\Psi'_{m,\lambda_m}$ be an arc in $\L(X',{\xi})$ such that $\mu_{\infty}(\Psi'_{m,\lambda_m})=\Psi_{m,\lambda_m}$ and so that 
$\varphi'_m\in \overline{\{\Psi'_{m,\lambda_m}\}}$.   

Note that for every $m$ we have
\begin{equation*}
m+c_1>\omega_{m}c_1=\nu_t(\varphi'_m)\geq\nu_t(\Psi'_{m, \lambda_m})\geq m
\end{equation*}
and
\begin{equation*}
\omega_{m}\cdot c_1\cdot a=\nu_t(\varphi'_m(IB))\geq \nu_t(\Psi'_{m, \lambda_m} (IB)).
\end{equation*}
Finally
\begin{equation*}
\frac{\omega_{m}\cdot c_1\cdot a}{m\cdot b}\geq \frac{\nu_t(\Psi'_{m, \lambda_m} (IB))}{m\cdot b} \geq
\frac{\nu_t(\Psi'_{m, \lambda_m} (IB))}{b\cdot\nu_t(\Psi'_{m, \lambda_m})}= \overline{r}_{X,\Psi'_{m, \lambda_m}} = \overline{r}_{X,\Psi_{m, \lambda_m}} \geq
\frac{a}{b}.
\end{equation*}
Now the result follows  by observing that: 
\begin{equation}
\frac{a}{b} \leq \lim_{n\to \infty} \delta_m\leq  \lim_{m\to \infty} \overline{r}_{X,\Psi_{m, \lambda_m}} \leq \lim_{m\to \infty} \left(\frac{\omega_{m}\cdot c_1\cdot a}{m\cdot b} \right) =\lim_{m\to\infty}\left(\frac{c_1\cdot a}{m\cdot b}\left\lceil\frac{m}{c_1}\right\rceil\right)=\frac{a}{b},
\end{equation} 
where the last equality follows by noticing that:   
\begin{equation*}
m+c_1\geq c_1\left\lceil\frac{m}{c_1}\right\rceil\geq m,
\end{equation*}
and that
\begin{equation*}
1+\frac{c_1}{m}=\frac{m+c_1}{m}\geq
\frac{c_1}{m}\left\lceil\frac{m}{c_1}\right\rceil\geq
\frac{c_1}{m}\frac{m}{c_1}=1.
\end{equation*}
\end{proof}

\section{On the values of the ${\mathbb Q}$-persistance}\label{valores}

The results in the previous sections are valid for varieties defined over a perfect field of arbitrary characteristic.
In this section we restrict to fields of characteristic zero since we use the existence of resolution of singularities.

\begin{Thm} \label{ThExistDivFat} 
	Let $X$ be a $d$-dimensional algebraic variety defined over a field of characteristic zero.
	Fix a point $\xi\in \Mm$ and let
	$\varphi \in \mathcal{L}(X, {\xi})$ be an arc such that $\varphi\not\in\mathcal{L}(\Sing(X))$.
	Then there exists a divisorial fat arc $\psi\in\mathcal{L}(X, {\xi})$ such that 
	\begin{itemize}
		\item $\varphi\in\overline{\{\psi\}}$   and
		\item $\bar{r}_{X,\varphi}=\bar{r}_{X,\psi}$.
	\end{itemize}
\end{Thm}

\begin{proof} 
Assume $X$ is affine. Recall that 
there is an \'etale morphims $\mu:X'\to X$ and a point $\xi'\in X'$ with $\mu(\xi')=\xi$, such that
the situation in \ref{setting} holds for $X'$.
This means that $X'=\Spec(B)$, there exists a finite morphism $\beta:X'\to V^{(d)}$ with $V^{(d)}=\Spec(S)$ smooth as in \ref{setting}. Moreover, as in (\ref{salvo_entera}), we have that, up to integral closure, $\mathcal{G}^{(d)}=S[IW^b]$.
There exists an arc $\varphi'\in\mathcal{L}(X', {\xi'})$ such that $\mu_{\infty}(\varphi')=\varphi$.

Let $\Pi:Y'\to X'$ be a simultaneous log-resolution of the ideals $I\mathcal{O}_{X'}$ and $\mathfrak{m}_{X',\xi'}$: 
\begin{equation} 
	I\mathcal{O}_{Y'}=I(H_1)^{a_1}\cdots I(H_N)^{a_N},
	\qquad
	\mathfrak{m}_{X',\xi'}\mathcal{O}_{Y'}=I(H_1)^{c_1}\cdots I(H_N)^{c_N}. 
\end{equation}
Note that  $V(I)=\Max\mult_{X'}\subset\Sing(X')$, and therefore,
since $\varphi'\not\in\L(\Sing(X'))$,    it factors through $Y'$ and  there is a unique $\varphi'_{Y'}\in\mathcal{L}(Y')$ such that   $\varphi'=\Pi_{\infty}(\varphi'_{Y'})$.  Set  $\ell_i=\nu_t(\varphi'_{Y'}(I(H_i)))$, $i=1,\ldots,N$, then
we have that
\begin{equation} \label{eq_r_barra_0}
	\bar{r}_{X,\varphi}=\bar{r}_{X',\varphi'}=  
	\frac{\nu_t(\varphi'(I\mathcal{O}_{X'}))}{b\nu_t(\varphi'(\mathfrak{m}_{X,\xi}\mathcal{O}_{X'}))}=
	\frac{\nu_t(\varphi'_{Y'}(I\mathcal{O}_{Y'}))}{b\nu_t(\varphi'_{Y'}(\mathfrak{m}_{X,\xi}\mathcal{O}_{Y'}))}=  
	\frac{\sum_{i=1}^{N}\ell_i a_i}{b\sum_{i=1}^{N}\ell_i c_i}.
\end{equation}

Consider the multi-index $\ell=(\ell_1,\ldots,\ell_N)$ and  the multi-contact loci
(\ref{Multi_Contact}) (see also \cite{E_L_M}) 
\begin{equation*}
	\Cont^{\ell}(E'_{Y'})=\left\{\phi\in\mathcal{L}(Y')\mid \nu_t(\phi(I(H_i)))=\ell_i,\ i=1,\ldots,N\right\}
\end{equation*}
	where $E'_{Y'}$ is the simple normal crossing divisor on $Y'$ with irreducible components $H_1,\ldots,H_N$.
	The set $\Cont^{\ell}(E'_{Y'})$ is irreducible and not empty since   $\varphi'_{Y'}\in\Cont^{\ell}(E'_{Y'})$.
	By \cite[2.6]{E_L_M} this set is divisorial.  Let $\psi'_{Y'}$ be the generic point of $\Cont^{\ell}(E'_{Y'})$.
	The arc $\psi'=\Pi_{\infty}(\psi'_{Y'})$ is a fat divisorial arc on $X'$, we have that  $\varphi'\in\overline{\{\psi'\}}$ and
	$\bar{r}_{X',\psi'}=\bar{r}_{X',\varphi'}$.
	Now set $\psi=\mu_{\infty}(\psi')$, we also have that $\varphi\in\overline{\{\psi\}}$ and by \cite[3.2]{I_Crelle} the arc $\psi$ is divisorial in $X$.

\end{proof}

\begin{Thm} \label{Values}   Let $X$ be a $d$-dimensional algebraic variety defined over a field of characteristic zero $k$ and let $\xi\in \Mm$. Let $\mu: X'\to X$ be an \'etale morphism   with $\mu(\xi')=\xi$ where $\G_{X'}$ is defined (see \ref{setting}), and 
assume that,  up to integral closure, $\mathcal{G}_{X'}={\mathcal O}_{X'}[IW^b]$ (see (\ref{salvo_entera})).
Let $\Pi:Y\to X'$ be a simultaneous log-resolution of the ideals $I$ and $\mathfrak{m}_{\xi'}$.
Denote by  $H_1,\ldots,H_N$ the irreducible components of the exceptional locus,
\begin{equation} \label{EqSimResIdMax}
\begin{aligned}
  I\mathcal{O}_{Y}=I(H_1)^{a_1}\cdots I(H_N)^{a_N}, \\
  \mathfrak{m}_{\xi'}\mathcal{O}_{Y}=I(H_1)^{c_1}\cdots I(H_N)^{c_N}.
\end{aligned} 
\end{equation} 
Set $\Lambda=\left\{i\in \left\{1,\ldots,N\right\}\mid a_i\neq 0\right\}$.
	Then, for any arc $\varphi$ in $\L(X, {\xi})$,   with $\varphi\not\in\L(\Sing(X))$,  
	\begin{equation*}
	\frac{1}{b}\min_{i\in\Lambda}\frac{a_i}{c_i}\leq \bar{r}_{X,\varphi}\leq \frac{1}{b}\max_{i\in\Lambda}\frac{a_i}{c_i}, 
	\end{equation*}
where we use the convention that $\frac{a_i}{c_i}=\infty$ whenever $c_i=0$ and $a_i\neq 0$.

Moreover
\begin{equation*}
	\frac{1}{b}\min_{i\in\Lambda}\frac{a_i}{c_i}=\inf\left\{ \bar{r}_{X,\varphi} \mid \varphi\in\mathcal{L}(X, {\xi}) \right\}
	\quad\text{and}\quad
	\frac{1}{b}\max_{i\in\Lambda}\frac{a_i}{c_i}=\sup\left\{ \bar{r}_{X,\varphi} \mid \varphi\in\mathcal{L}(X, {\xi}) \right\}.
	\end{equation*}
\end{Thm}

\begin{proof}
The first inequalities are a consequence of (\ref{eq_r_barra_0}), 
\begin{equation*}
	\frac{1}{b}\min_{i\in\Lambda}\frac{a_i}{c_i}\leq \bar{r}_{X,\varphi}=
	\frac{1}{b}\frac{\sum_{i=1}^{N}\ell_i a_i}{\sum_{i=1}^{N}\ell_i c_i}
	\leq\frac{1}{b} \max_{i\in\Lambda}\frac{a_i}{c_i}.
\end{equation*}
We only need to study the case when some $c_i=0$.
In this case the maximum value has to be $\infty$ and we claim that there are arcs $\varphi$ such that $\bar{r}_{X,\varphi}$ is bigger than any positive real number.
	
Assume that $c_1=0$ and $a_1\neq 0$. There exist some $c_j\neq 0$, after reordering the indexes assume that $c_2\neq 0$.
	
Set $\ell_n=(n,1,0,\ldots,0)$ and let $\psi_n$ the generic point of $\Cont^{\ell_n}(E_Y)$.
Note that
\begin{equation*}
	\lim_{n\to\infty}\bar{r}_{X,\psi_n}=\lim_{n\to\infty}\frac{n a_1+a_2}{bc_2}=\infty
\end{equation*}
\end{proof}


\begin{thebibliography}{10}

\bibitem{Abad}
C.~Abad.
\newblock On the highest multiplicity locus of algebraic varieties and {R}ees
  algebras.
\newblock {\em J. Algebra}, 441:294--313, 2015.

\bibitem{COA} C.~Abad, A.~Bravo, O.E.~Villamayor U.   
\newblock Finite morphisms and simultaneous reduction of the multiplicity. 
\newblock {\em Math. Nachr.}, 293:8--38, 2020.

\bibitem{Abhy1}
S.~Abhyankar.
\newblock Corrections to ``{L}ocal uniformization on algebraic surfaces over
  ground fields of characteristic {$p\not=0$}''.
\newblock {\em Ann. of Math. (2)}, 78:202--203, 1963.

\bibitem{Abhy2}
S.~Abhyankar.
\newblock Uniformization in {$p$}-cyclic extensions of algebraic surfaces over
  ground fields of characteristic {$p$}.
\newblock {\em Math. Ann.}, 153:81--96, 1964.

\bibitem{Benito_V}
A.~Benito and O.~E. Villamayor~U.
\newblock Techniques for the study of singularities with applications to
  resolution of 2-dimensional schemes.
\newblock {\em Math. Ann.}, 353(3):1037--1068, 2012.

\bibitem{Bhatt}
Bhargav Bhatt.
\newblock Algebraization and {T}annaka duality.
\newblock {\em Camb. J. Math.}, 4(4):403--461, 2016.

\bibitem{B-M}
E.~Bierstone and P.~D. Milman.
\newblock Canonical desingularization in characteristic zero by blowing up the
  maximum strata of a local invariant.
\newblock {\em Invent. Math.}, 128(2):207--302, 1997.

\bibitem{Bosch}
Siegfried Bosch.
\newblock {\em Algebraic geometry and commutative algebra}.
\newblock Universitext. Springer, London, 2013.

\bibitem{BEP2}
A.~Bravo, E.~Encinas, and B.~Pascual-Escudero.
\newblock Nash multiplicity sequences and {H}ironaka's order function.
\newblock {\em Preprint, arXiv:1802.02566 [math.AG]}, 2018.
To appear in Indiana Univ. Math. J.

\bibitem{Br_E_P-E}
A.~Bravo, S.~Encinas, and B.~Pascual-Escudero.
\newblock Nash multiplicities and resolution invariants.
\newblock {\em Collectanea Mathematica}, 68(2):175--217, 2017.

\bibitem{Br_E_V}
A.~Bravo, S.~Encinas, and O.~Villamayor~U.
\newblock A simplified proof of desingularization and applications.
\newblock {\em Rev. Mat. Iberoamericana}, 21(2):349--458, 2005.

\bibitem{Br_G-E_V}
A.~Bravo, M.~L. Garcia-Escamilla, and O.~E. Villamayor~U.
\newblock On {R}ees algebras and invariants for singularities over perfect
  fields.
\newblock {\em Indiana Univ. Math. J.}, 61(3):1201--1251, 2012.

\bibitem{Br_V}
A.~Bravo and O.~Villamayor~U.
\newblock Singularities in positive characteristic, stratification and
  simplification of the singular locus.
\newblock {\em Adv. Math.}, 224(4):1349--1418, 2010.

\bibitem{Br_V1}
A.~Bravo and O.~E. Villamayor~U.
\newblock Elimination algebras and inductive arguments in resolution of
  singularities.
\newblock {\em Asian J. Math.}, 15(3):321--355, 2011.

\bibitem{Br_V2}
A.~Bravo and O.~E. Villamayor~U.
\newblock On the behavior of the multiplicity on schemes: stratification and
  blow ups.
\newblock In {\em The resolution of singular algebraic varieties}, pages
  81--207. Amer. Math. Soc., Providence, RI, 2014.

\bibitem{Cos_Pilt1}
V.~Cossart and O.~Piltant.
\newblock Resolution of singularities of threefolds in positive characteristic.
  {I}. {R}eduction to local uniformization on {A}rtin-{S}chreier and purely
  inseparable coverings.
\newblock {\em J. Algebra}, 320(3):1051--1082, 2008.

\bibitem{Cos_Pilt2}
V.~Cossart and O.~Piltant.
\newblock Resolution of singularities of threefolds in positive characteristic.
  {II}.
\newblock {\em J. Algebra}, 321(7):1836--1976, 2009.

\bibitem{Cut}
S.~D. Cutkosky.
\newblock Resolution of singularities for 3-folds in positive characteristic.
\newblock {\em Amer. J. Math.}, 131(1):59--127, 2009.

\bibitem{Dade}
E.~D. Dade.
\newblock {\em Multiplicity and monoidal transformations}.
\newblock Princeton University, 1960.
\newblock Thesis (Ph.D.).

\bibitem{deF_Doc}
T.~de~Fernex and R.~Docampo.
\newblock Terminal valuations and the {N}ash problem.
\newblock {\em Invent. Math.}, 203(1):303--331, 2016.

\bibitem{deF_E_I}
T.~de~Fernex, L.~Ein, and S.~Ishii.
\newblock Divisorial valuations via arcs.
\newblock {\em Publ. Res. Inst. Math. Sci.}, 44(2):425--448, 2008.

\bibitem{E_L_M}
L.~Ein, R.~Lazarsfeld, and M.~Musta{\c{t}}{\u{a}}.
\newblock Contact loci in arc spaces.
\newblock {\em Compos. Math.}, 140(5):1229--1244, 2004.

\bibitem{E_Hau}
S.~Encinas and H.~Hauser.
\newblock Strong resolution of singularities in characteristic zero.
\newblock {\em Comment. Math. Helv.}, 77(4):821--845, 2002.

\bibitem{E_V97}
S.~Encinas and O.~Villamayor.
\newblock A course on constructive desingularization and equivariance.
\newblock In {\em Resolution of singularities ({O}bergurgl, 1997)}, volume 181
  of {\em Progr. Math.}, pages 147--227. Birkh\"auser, Basel, 2000.

\bibitem{E_V}
S.~Encinas and O.~Villamayor.
\newblock Rees algebras and resolution of singularities.
\newblock In {\em Proceedings of the {XVI}th {L}atin {A}merican {A}lgebra
  {C}olloquium ({S}panish)}, Bibl. Rev. Mat. Iberoamericana, pages 63--85. Rev.
  Mat. Iberoamericana, Madrid, 2007.

\bibitem{Hickel05}
M.~Hickel.
\newblock Sur quelques aspects de la g\'eom\'etrie de l'espace des arcs
  trac\'es sur un espace analytique.
\newblock {\em Ann. Fac. Sci. Toulouse Math. (6)}, 14(1):1--50, 2005.

\bibitem{Hir}
H.~Hironaka.
\newblock Resolution of singularities of an algebraic variety over a field of
  characteristic zero. {I}, {II}.
\newblock {\em Ann. of Math. (2) {79} (1964), 109--203; ibid. (2)},
  79:205--326, 1964.

\bibitem{Hir1}
H.~Hironaka.
\newblock Idealistic exponents of singularity.
\newblock In {\em Algebraic geometry ({J}. {J}. {S}ylvester {S}ympos., {J}ohns
  {H}opkins {U}niv., {B}altimore, {M}d., 1976)}, pages 52--125. Johns Hopkins
  Univ. Press, Baltimore, Md., 1977.

\bibitem{I_Crelle}
S.~Ishii.
\newblock Arcs, valuations and the {N}ash map.
\newblock {\em J. Reine Angew. Math.}, 588:71--92, 2005.

\bibitem{I08}
S.~Ishii.
\newblock Maximal divisorial sets in arc spaces.
\newblock In {\em Algebraic geometry in {E}ast {A}sia---{H}anoi 2005},
  volume~50 of {\em Adv. Stud. Pure Math.}, pages 237--249. Math. Soc. Japan,
  Tokyo, 2008.

\bibitem{I2}
S.~Ishii.
\newblock Smoothness and jet schemes.
\newblock In {\em Singularities---{N}iigata--{T}oyama 2007}, volume~56 of {\em
  Adv. Stud. Pure Math.}, pages 187--199. Math. Soc. Japan, Tokyo, 2009.

\bibitem{I_K}
S.~Ishii and J.~Koll\'ar.
\newblock The {N}ash problem on arc families of singularities.
\newblock {\em Duke Math. J.}, 120(3):601--620, 2003.

\bibitem{K_M1}
Hiraku Kawanoue and Kenji Matsuki.
\newblock Resolution of singularities of an idealistic filtration in dimension
  3 after {B}enito-{V}illamayor.
\newblock In {\em Minimal models and extremal rays ({K}yoto, 2011)}, volume~70
  of {\em Adv. Stud. Pure Math.}, pages 115--214. Math. Soc. Japan, [Tokyo],
  2016.

\bibitem{L-J}
M.~Lejeune-Jalabert.
\newblock Courbes trac\'ees sur un germe d'hypersurface.
\newblock {\em Amer. J. Math.}, 112(4):525--568, 1990.

\bibitem{L-J_Mou_Re}
M.~Lejeune-Jalabert, H.~Mourtada, and A.~Reguera.
\newblock Jet schemes and minimal embedded desingularization of plane branches.
\newblock {\em Rev. R. Acad. Cienc. Exactas F\'\i s. Nat. Ser. A Math. RACSAM},
  107(1):145--157, 2013.

\bibitem{Lipman3}
J.~Lipman.
\newblock Desingularization of two-dimensional schemes.
\newblock {\em Ann. Math. (2)}, 107(1):151--207, 1978.

\bibitem{Milne}
J.~S. Milne.
\newblock {\em \'Etale cohomology}, volume~33 of {\em Princeton Mathematical
  Series}.
\newblock Princeton University Press, Princeton, N.J., 1980.

\bibitem{Mou4}
H.~Mourtada.
\newblock Jet schemes of rational double point singularities.
\newblock In {\em Valuation theory in interaction}, EMS Ser. Congr. Rep., pages
  373--388. Eur. Math. Soc., Z\"urich, 2014.

\bibitem{Mus1}
M.~Musta{\c{t}}{\u{a}}.
\newblock Jet schemes of locally complete intersection canonical singularities.
\newblock {\em Invent. Math.}, 145(3):397--424, 2001.
\newblock With an appendix by David Eisenbud and Edward Frenkel.

\bibitem{Nash}
J.~F. Nash, Jr.
\newblock Arc structure of singularities.
\newblock {\em Duke Math. J.}, 81(1):31--38 (1996), 1995.
\newblock A celebration of John F. Nash, Jr.

\bibitem{P-ET}
B.~Pascual-Escudero.
\newblock Algorithmic resolution of singularities and {N}ash multiplicities
  sequences.
\newblock Universidad Aut\'onoma de Madrid, 2018. Thesis (Ph.D.).

\bibitem{P-E2}
B.~Pascual-Escudero.
\newblock Nash multiplicities and isolated points of maximal multiplicity.
\newblock {\em J. Pure and Applied
	Algebra}, 223(6):2598--2615, 2019.

\bibitem{Re3}
A.-J. Reguera.
\newblock Families of arcs on rational surface singularities.
\newblock {\em Manuscripta Math.}, 88(3):321--333, 1995.

\bibitem{Su_V}
Diego Sulca and Orlando Villamayor.
\newblock An introduction to resolution of singularities via the multiplicity.
\newblock In {\em Singularities, Algebraic Geometry, Commutative Algebra, and
  Related Topics}, pages 263--317. Springer, 2018.

\bibitem{V1}
O.~Villamayor.
\newblock Constructiveness of {H}ironaka's resolution.
\newblock {\em Ann. Sci. \'Ecole Norm. Sup. (4)}, 22(1):1--32, 1989.

\bibitem{V2}
O.~Villamayor.
\newblock Patching local uniformizations.
\newblock {\em Ann. Sci. \'Ecole Norm. Sup. (4)}, 25(6):629--677, 1992.

\bibitem{V00}
O.~E. Villamayor.
\newblock Tschirnhausen transformations revisited and the multiplicity of the
  embedded hypersurface.
\newblock {\em Bol. Acad. Nac. Cienc. (C\'ordoba)}, 65:233--243, 2000.
\newblock Colloquium on Homology and Representation Theory (Spanish)
  (Vaquer{\'{\i}}as, 1998).

\bibitem{V07}
O.~Villamayor~U.
\newblock Hypersurface singularities in positive characteristic.
\newblock {\em Adv. Math.}, 213(2):687--733, 2007.

\bibitem{V3}
O.~Villamayor~U.
\newblock Rees algebras on smooth schemes: integral closure and higher
  differential operator.
\newblock {\em Rev. Mat. Iberoam.}, 24(1):213--242, 2008.

\bibitem{V}
O.~E. Villamayor~U.
\newblock Equimultiplicity, algebraic elimination, and blowing-up.
\newblock {\em Adv. Math.}, 262:313--369, 2014.

\bibitem{Villamayor2005_2}
Orlando~E. Villamayor~U.
\newblock Differential operators on smooth schemes and embedded singularities.
\newblock {\em Rev. Un. Mat. Argentina}, 46(2):1--18 (2006), 2005.

\bibitem{Vojta}
P.~Vojta.
\newblock Jets via {H}asse-{S}chmidt derivations.
\newblock In {\em Diophantine geometry}, volume~4 of {\em CRM Series}, pages
  335--361. Ed. Norm., Pisa, 2007.

\bibitem{Z-SII}
O.~Zariski and P.~Samuel.
\newblock {\em Commutative algebra. {V}ol. {II}}.
\newblock The University Series in Higher Mathematics. D. Van Nostrand Co.,
  Inc., Princeton, N. J.-Toronto-London-New York, 1960.

\end{thebibliography}
\end{document}